\documentclass[final,1p,times]{elsarticle}

\usepackage{amsfonts}
\usepackage{graphicx}
\usepackage{amsmath}
\usepackage{graphicx}
\usepackage{amssymb}
\usepackage[colorlinks]{hyperref}
\usepackage{color}
\usepackage{xcolor}

\newtheorem{theorem}{Theorem}

\newtheorem{corollary}[theorem]{Corollary}

\newtheorem{definition}[theorem]{Definition}

\newtheorem{lemma}[theorem]{Lemma}

\newtheorem{remark}[theorem]{Remark}
\newenvironment{proof}[1][Proof]{\textbf{#1.} }{\ \rule{0.5em}{0.5em}}

\usepackage{amssymb}

\newcommand{\correction}[2]{#2}
\newcommand{\cor}[1]{#1}

\newcommand{\mymarginpar}[1]{}

\journal{Communications in Nonlinear Science and Numerical Simulation}

\begin{document}
\begin{frontmatter}



\title{
Computer Assisted Proof of Drift Orbits Along Normally Hyperbolic Manifolds II: Application to the Restricted Three Body Problem
}


\author[agh]{Maciej J. Capi\'nski \fnref{mjcA}}
\address[agh]{Faculty of Applied Mathematics, AGH University of Science and Technology, al. Mickiewicza 30, 30-059 Krak\'ow, Poland.}
\cortext[cor]{Corresponding author.}
\fntext[mjcA]{Partially supported by the NCN grants 2018/29/B/ST1/00109, 2019/35/B/ST1/00655.}

\author[agh]{Natalia Wodka \corref{cor}\fnref{mjcA}}

\begin{abstract}
We present a computer assisted proof of diffusion in the Planar Elliptic Restricted Three Body Problem. We treat the elliptic problem as a perturbation of the circular problem, where the perturbation parameter is the eccentricity of the primaries. The unperturbed system preserves energy, and we show that for sufficiently small perturbations we have orbits with explicit  energy changes, independent from the size of the perturbation. 
The result is based on shadowing of orbits along transversal intersections of stable/unstable manifolds of a normally hyperbolic cylinder.
\end{abstract}


\end{frontmatter}



\section{Introduction}


Autonomous Hamiltonian systems have an integral of motion given by the Hamiltonian. We shall refer to this integral as the energy. Our question is whether for an arbitrarily small perturbation we can have orbits, which change in energy by a prescribed constant independent form the size of the perturbation.

In our case, the system of interest is the Planar Restricted Three Body Problem that describes the motion of a small massless particle under the gravitational pull of two large bodies. They rotate on Keplerian orbits and we call them primaries. The motion of the massless particle is on the same plane as the primaries. When the Keplerian orbit is circular, then in the coordinate frame which rotates with the primaries we have an autonomous Hamiltonian system;  the Planar Circular Restricted Three Body Problem (PCR3BP). When the Keplerian orbits are elliptic, then we deal with the Elliptic Problem (PER3BP), which is a time dependent Hamiltonian system.

The system we study has a normally hyperbolic invariant manifold (NHIM) prior to the perturbation, with stable and unstable manifolds which intersect transversally. This means that we consider an `apriori-chaotic' system. (This is a much simpler problem than starting with a fully integrable system; as in the famous example of Arnold  \cite{MR0163026}.) Our diffusion mechanism is based on establishing the existence of trajectories that shadow the intersections of the stable and unstable manifolds and change energy under the influence of the perturbation.

Our result is based on the geometric and shadowing results of Delshams, Gidea, de la Llave and Seara. At the core of the mechanism is the scattering map \cite{DLS1}; that is a map which for a point on the NHIM  assigns  another point from the NHIM, if their unstable and stable fibres intersect in a nontrivial way. The benefit of studying a scattering map is twofold. 
Firstly, one can exploit perturbative techniques to study the effect of the parameter on the scattering map. This is typically done by using Melnikov integrals in the case of continuous systems, or Melnikov sums in the case of discrete systems \cite{DLS1, MR1373998}. Once the effects of the perturbation on the scattering map is established, the second benefit is that there exist true orbits of the system, which will shadow the `pseudo orbits' of the scattering map \cite{GLS}. This means that if one can prove existence of `pseudo orbits' of the scattering map which have macroscopic changes in energy, then this ensures the existence of true trajectories of the system that will do the same.

The strategy for proving diffusion using this scheme is described in  \cite{DLS1, MR1373998,GLS,MR4160091}. 
For our proof we will use the results from \cite{Jorge}, which provides a formulation for these methods, that can be implemented to obtain a computer assisted proof. The main feature of  \cite{Jorge} is that the assumptions needed for the diffusion mechanism can be validated in a finite number of steps by checking various bounds on the properties of the system.

In our case the components that need to be established are the following. We  establish bounds on the family of Lyapunov orbits in the PCR3BP, which constitute our NHIM prior to the perturbation. Next we establish bounds for the local stable/unstable manifolds of the NHIM, and prove that they intersect transversally. We  ensure that the scattering map is well defined and obtain explicit bounds on its dynamics. We also validate twist conditions, which ensure that the NHIM after perturbation contains a Cantor set of KAM tori. Then we can apply the theorems from \cite{Jorge}  to prove diffusion. All the above steps can be established with the assistance of rigorous, interval arithmetic based estimates, performed with the aid of a computer. 

For proofs of diffusion in the PER3BP a computer assisted approach is not strictly necessary. The paper \cite{MR3927089} provides an analytic proof of diffusion in the PER3BP. This result also used the scattering map theory and shadowing as the mechanism. 
To obtain the analytic proof   though
\cite{MR3927089} required that the mass of one of the primaries is sufficiently small, and that the angular momentum of the massless particle is sufficiently large.
The difference with the result from the current paper is that we can work with the explicit mass of Jupiter and with the explicit energy of a massless particle, which corresponds to that of the comet Oterma. As in \cite{MR3927089} we have to assume that the eccentricity of the system is sufficiently small. 

A recent result which works with explicit masses, explicit energy and explicit interval of eccentricities that starts from zero and reaches a physical value is given in \cite{CG}. This is based on a computer assisted construction of `correctly aligned windows'. The difference is that in the current paper we only need to establish the intersections of the manifolds for the unperturbed system, and check some explicit conditions from \cite{Jorge}. The shadowing of orbits is then automatically taken care of by \cite{GLS}, and we do not need to carry out the explicit construction of windows as in \cite{CG}. Our results are weaker though than \cite{CG}.
\correction{comment 17}{}

The geometric setup for our method follows very closely that from  \cite{MR2785975, MR3604613}. There are some small differences though, the main being that we use Poincar\'e maps and compute finite Melnikov sums, instead of Melnikov integrals. Our original plan was to follow directly the setup from \cite{MR3604613}, but we have found that computing sharp bounds on integrals along trajectories is more difficult than computing bounds on sums along discrete trajectories of Poincar\'e maps. \correction{comment 1}{
The key difference between the current result and \cite{MR3604613} is that in \cite{MR3604613} some of the assumptions of the theorems which ensure the diffusion were validated by using non-rigorous numerics.  In this paper we conduct a proof by means of rigorous, interval arithmetic based validation. For this we need to provide a computer assisted proof of the existence of the normally hyperbolic manifold and the twist condition needed for the application of the KAM theorem on that manifold. We also provide a computer assisted proof that the scattering is well defined, which involves establishing transversal intersections of its stable/unstable manifolds. Finally, we provide rigorous interval arithmetic bounds on the Melnikov sums, which leads to diffusion.
}

Our tool for the computer assisted implementation\footnote{\cor{The code for the computer assisted proof is available on the personal web page of M. J. Capi\'nski.}}\correction{comment 4}{} is the CAPD\footnote{Computer Assisted Proofs in Dynamics: http://capd.ii.uj.edu.pl} library \cite{capd}. The tools which we use are quite standard: The existence of the family of Lyapunov orbits as well as the proof of their homoclinic orbits is done by exploiting symmetry properties of the PCR3BP, combined with parallel shooting and the Krawczyk method. To establish bounds on the stable and unstable manifolds and the transversality of their intersections needed to ensure that the scattering map is well defined we use cones. The assumptions of theorems \cite{Jorge} leading to diffusion can then be checked by computing sums along finite fragments of homoclinic orbits.

The paper is organised as follows. Section \ref{sec:prel} contains preliminaries, which include the Krawczyk method, introduction to the scattering map and the shadowing results for scattering maps, as well as a short introduction to  the restricted three body problem. In section \ref{sec:main} we state the main result of the paper, which is written in Theorem \ref{th:main}. Section \ref{sec:proof} contains the proof of the main theorem. Some of the technical issues are delegated to the Appendix.


\section{Preliminaries\label{sec:prel}}

\subsection{Notations}

For a set $U$ in a topological space we will denote its interior by
$\mathrm{int}U$, its closure by $\bar{U}$ and its boundary by $\partial U$.

We shall denote identity by $Id$. For a point $p$ expressed in some
coordinates $p=\left(  x,y\right)  $ we shall denote by $\pi_{x}\mathbf{\ }%
$and $\pi_{y}$ the projections onto the given coordinates, i.e. $\pi_{x}p=x$
and $\pi_{y}p=y$. For $p=\left(  p_{1},\ldots,p_{n}\right)  $ we will write
$\pi_{i}$ to denote the projection onto the $i$-th coordinate, i.e. $\pi
_{i}p=p_{i}$.

We write $\left\Vert \cdot\right\Vert _{\max}$ for the maximum norm in
$\mathbb{R}^{n}$, and for a matrix $A$ write $\left\Vert A\right\Vert _{\max}$
for the matrix norm induced by $\left\Vert \cdot\right\Vert _{\max}$.

We shall denote a $k$-dimensional torus as $\mathbb{T}^{k}$, with the
convention that $\mathbb{T}:=\mathbb{R}\,\mathrm{mod}\,2\pi$.

\subsection{Krawczyk method}

We refer to a cartesian product of intervals as an interval set. For a set
$U\subset\mathbb{R}^{n}$ we shall denote by $\left[  U\right]  $ an interval
set such that $U\subset\left[  U\right]  .$ We refer to such set as an
interval enclosure of $U$. An interval enclosure is not unique. In our
applications we shall consider interval enclosures of objects of interest;
namely: fixed points, periodic orbits, homoclinic orbits and invariant
manifolds. The smaller the interval enclosure is, the more accurate is our
bound on the object of interest. We shall refer to a matrix whose coefficients are intervals as an interval matrix.


For a $C^{1}$ function $F:\mathbb{R}^{n}\rightarrow\mathbb{R}^{n}$ we denote
by $\left[  DF\left(  X\right)  \right]  \subset\mathbb{R}^{n\times n}$ the
interval matrix enclosure of the derivatives of $F$ on $X$, namely we
consider
\[
\left[  DF\left(  X\right)  \right]  =\left\{  A=\left(  a_{i,j}\right)
:a_{ij}\in\left[  \inf_{x\in X}\frac{\partial f_{i}}{\partial x_{j}}\left(
x\right)  ,\sup_{x\in X}\frac{\partial f_{i}}{\partial x_{j}}\left(  x\right)
\right]  \right\}  .
\]
Below theorem, known as the Krawczyk method, can be used to establish bounds
on zeros of functions.

\begin{theorem}
\label{th:Krawczyk}\cite{MR1318950}Let $F:\mathbb{R}^{n}\rightarrow
\mathbb{R}^{n}$ be a $C^{1}$ function. Let $X\subset\mathbb{R}^{n}$ be an
interval set, let $x\in X$, let $C\in\mathbb{R}^{n\times n}$ be a linear
isomorphism and let
\[
K\left(  x,X,F\right)  :=x-CF\left(  x\right)  +\left(  Id-C\left[  DF\left(
X\right)  \right]  \right)  \left(  X-x\right)  .
\]
If%
\[
K\left(  x,X,F\right)  \subset\mathrm{int}X,
\]
then there exists a unique point $x^{\ast}\in X$ for which
\[
F\left(  x^{\ast}\right)  =0.
\]

\end{theorem}

\subsection{Normally hyperbolic invariant manifolds and the scattering map}
In this section we recall the results which we use for our diffusion mechanism. We follow the setup from \cite{Jorge}, which is based on the scattering map theory described in \cite{DLS1}, and the diffusion mechanism from \cite{GLS}, which is the main tool for our proof.

\begin{definition}
\label{def:nhim} 
Let $M$ be a smooth  $n$-dimensional \correction{comment 18}{Riemannian }manifold, and let  $f : M \to M$ be a $C^r$ diffeomorphism, with $r> 1$.
Let $\Lambda \subset M$ be a compact manifold
without boundary, invariant under $f$, i.e., $f(\Lambda )=\Lambda $ . We say that $\Lambda $ is a normally hyperbolic invariant
manifold (with symmetric rates) if there exists a constant $C>0,$ rates $%
0<\lambda <\mu ^{-1}<1$ and a $Tf$-invariant splitting for every $x\in
\Lambda $%
\begin{equation*}
T_x M=E_{x}^{u}\oplus E_{x}^{s}\oplus T_{x}\Lambda
\end{equation*}%
such that%
\begin{align}
v& \in E_{x}^{u}\Leftrightarrow \left\Vert Df^{k}(x)v\right\Vert \leq
C\lambda ^{-k}\left\Vert v\right\Vert ,\quad k\leq 0,
\label{eq:rate-cond-nhim1} \\
v& \in E_{x}^{s}\Leftrightarrow \left\Vert Df^{k}(x)v\right\Vert \leq
C\lambda ^{k}\left\Vert v\right\Vert ,\quad k\geq 0,
\label{eq:rate-cond-nhim2} \\
v& \in T_{x}\Lambda \Rightarrow \left\Vert Df^{k}(x)v\right\Vert \leq C\mu
^{|k|}\left\Vert v\right\Vert ,\quad k\in \mathbb{Z}.
\label{eq:rate-cond-nhim3}
\end{align}
\end{definition}

Let $d\left( x,\Lambda \right) $ stand for the distance between a point $x$
and the manifold $\Lambda $. 

Given a normally
hyperbolic invariant manifold and a suitable small tubular neighbourhood $%
U\subset M$ of $\Lambda $ one defines its local unstable and
local stable manifold \cite{MR292101} as%
\begin{align*}
W_{\Lambda }^{u}\left( f,U\right) & =\left\{ y\in M\,|\,f^{k}(y)\in U,d\left( f^{k}(y),\Lambda \right) \leq C_{y}\lambda
^{\left\vert k\right\vert },\,k\leq 0\right\} , \\
W_{\Lambda }^{s}\left( f,U\right) & =\left\{ y\in M\,|\,f^{k}(y)\in U,d\left( f^{k}(y),\Lambda \right) \leq C_{y}\lambda
^{k},\,k\geq 0\right\} ,
\end{align*}%
where $C_{y}$ is a positive constant, which can depend on $y$. We define the
(global) unstable and stable manifolds as%
\begin{equation*}
W_{\Lambda }^{u}\left( f\right) =\bigcup_{n\geq 0}f^{n}\left( W_{\Lambda
}^{u}\left( f,U\right) \right) ,\qquad W_{\Lambda }^{s}\left( f\right)
=\bigcup_{n\geq 0}f^{-n}\left( W_{\Lambda }^{s}\left( f,U\right) \right) .
\end{equation*}

The manifolds $W_{\Lambda }^{u}\left( f,U\right) $, $W_{\Lambda }^{s}\left(
f,U\right) $, $W_{\Lambda }^{u}\left( f\right) $ and $W_{\Lambda }^{s}\left(
f\right) $ are foliated by%
\begin{align*}
W_{x}^{u}\left( f,U\right) & =\left\{ y\in M\,|\,f^{k}(y)\in
U,d(f^{k}(y),f^{k}(x))\leq C_{x,y}\lambda ^{\left\vert k\right\vert
},\,k\leq 0\right\} , \\
W_{x}^{s}\left( f,U\right) & =\left\{ y\in M\,|\,f^{k}(y)\in
U,d(f^{k}(y),f^{k}(x))\leq C_{x,y}\lambda ^{k},\,k\geq 0\right\} ,
\end{align*}%
where $x\in \Lambda $ and $C_{x,y}$ is a positive constant, which can depend
on $x$ and $y$,%
\begin{equation*}
W_{x}^{u}\left( f\right) =\bigcup_{n\geq 0}f^{n}\left( W_{f^{-n}\left(
x\right) }^{u}\left( f,U\right) \right) ,\qquad W_{x}^{s}\left( f\right)
=\bigcup_{n\geq 0}f^{-n}\left( W_{f^{n}\left( x\right) }^{s}\left(
f,U\right) \right) .
\end{equation*}

Let%
\begin{equation}
l<\min \left\{ r,\frac{|\log \lambda |}{\log \mu }\right\} .
\label{eq:nhim-smoothness}
\end{equation}%
The manifold $\Lambda $ is $C^{l}$ smooth, the manifolds $W_{\Lambda
}^{u}\left( f\right) ,W_{\Lambda }^{s}\left( f\right) $ are $C^{l-1}$ and $%
W_{x}^{u}\left( f\right) $, $W_{x}^{s}\left( f\right) $ are $C^{r}$ \cite%
{DLS1}. Normally hyperbolic manifolds, as well as their stable and unstable manifolds and their fibres persist under small perturbations \correction{comment 11}{\cite{MR292101,MR426056,MR343314}.}

From now let $(M,\omega)$ be a smooth symplectic manifold. Let us assume that $\Lambda\subset M$ is a normally hyperbolic invariant manifold for 
a $C^r$ symplectic map $f:M\to M$, where $r>1$. We assume that $\Lambda$ is even dimensional and symplectic with the symplectic form $\omega|_{\Lambda}$, and that $f|_{\Lambda}$ is symplectic on $\Lambda$. 
We define two maps, 
\begin{align*}
\Omega_{+} & :W_{\Lambda}^{s}(f)\rightarrow\Lambda, \\
\Omega_{-} & :W_{\Lambda}^{u}\left( f\right) \rightarrow\Lambda,
\end{align*}
where $\Omega_{+}(x)=x_{+}$ iff $x\in W_{x_{+}}^{s}\left( f\right) $, and $%
\Omega_{-}(x)=x_{-}$ iff $x\in W_{x_{-}}^{u}\left( f\right) .$ These are
referred to as the wave maps.

\begin{definition}
\label{def:homoclinic-channel} We say that a manifold $\Gamma \subset
W_{\Lambda }^{u}\left( f\right) \cap W_{\Lambda }^{s}\left( f\right) $ is a
homoclinic channel for $\Lambda $ if the following conditions hold:

\begin{itemize}
\item[(i)] \label{itm:homoclinic-channel-c1} for every $x\in \Gamma $ 
\begin{align}
T_{x}W_{\Lambda }^{s}\left( f\right) + T_{x}W_{\Lambda }^{u}\left(
f\right) & =T_x M,  \label{eq:scatter-cond-1} \\
T_{x}W_{\Lambda }^{s}\left( f\right) \cap T_{x}W_{\Lambda }^{u}\left(
f\right) & =T_{x}\Gamma ,  \label{eq:scatter-cond-2}
\end{align}

\item[(ii)] \label{itm:homoclinic-channel-c2} the fibres of $\Lambda$
intersect $\Gamma$ transversally in the following sense%
\begin{align}
T_{x}\Gamma\oplus T_{x}W_{x_{+}}^{s}\left( f\right) & =T_{x}W_{\Lambda
}^{s}\left( f\right) ,  \label{eq:scatter-cond-3} \\
T_{x}\Gamma\oplus T_{x}W_{x_{-}}^{u}\left( f\right) & =T_{x}W_{\Lambda
}^{u}\left( f\right) ,  \label{eq:scatter-cond-4}
\end{align}
for every $x\in\Gamma$,

\item[(iii)] the wave maps $(\Omega_{\pm})|_{\Gamma}:\Gamma\rightarrow
\Lambda$ are diffeomorphisms onto their image.
\end{itemize}
\end{definition}

\begin{definition}\label{def:scattering-map}
Assume that $\Gamma$ is a homoclinic channel for $\Lambda$ and let 
\begin{equation*}
\Omega_{\pm}^{\Gamma}:=\left( \Omega_{\pm}\right) |_{\Gamma}.
\end{equation*}
We define a scattering map $\sigma^{\Gamma}$ for the homoclinic channel $\Gamma$ as 
\begin{equation*}
\sigma^{\Gamma}:=\Omega_{+}^{\Gamma}\circ\left( \Omega_{-}^{\Gamma}\right)
^{-1}:\Omega_{-}^{\Gamma}\left( \Gamma\right) \rightarrow\Omega_{+}^{\Gamma
}\left( \Gamma\right) .
\end{equation*}
\end{definition}
We have the following theorem, which is the basis for the diffusion mechanism from the subsequent section.
\begin{theorem} \cite{GLS}\label{th:shadowing}
\correction{comment 2}{
Assume that $f:\mathbb{R}^{n}\rightarrow 
\mathbb{R}^{n}$ is a sufficiently smooth map, $\Lambda\subset\mathbb{R}^{n}$
is a normally hyperbolic invariant manifold with stable and unstable
manifolds which intersect transversally along a homoclinic channel $\Gamma
\subset\mathbb{R}^{n},$ and $\sigma$ is the scattering map associated to $%
\Gamma$.}\correction{comment 19}{}

\cor{Assume that $f$ preserves measure absolutely continuous with respect to the
Lebesgue measure on $\Lambda$, and that $\sigma$ sends positive measure sets
to positive measure sets.

Let $m_{1},\ldots,m_{n}\in\mathbb{N}$ be a fixed sequence of integers. Let $%
\left\{ x_{i}\right\} _{i=0,\ldots,n}$ be a finite pseudo-orbit in $\Lambda$%
, of the form%
\begin{equation}
x_{i+1}=f^{m_{i}}\circ\sigma\left( x_{i}\right) ,\qquad i=0,\ldots
,n-1,\,n\geq1,   \label{eq:pseudo-orbit}
\end{equation}
that is contained in some open set $\mathcal{U}\subset\Lambda$ with almost
every point of $\mathcal{U}$ recurrent for $f|_{\Lambda}$. (The points $%
\left\{ x_{i}\right\} _{i=0,\ldots,n}$ do not have to be themselves
recurrent.)

Then for every $\delta>0$ there exists an orbit $\left\{ z_{i}\right\}
_{i=0,\ldots,n}$ of $f$ in $\mathbb{R}^{n}$, with $z_{i+1}=f^{k_{i}}\left(
z_{i}\right) $ for some $k_{i}>0$, such that $d\left( z_{i},x_{i}\right)
<\delta$ for all $i=0,\ldots,n$.
}
\end{theorem}

\begin{remark}\label{rem:finite-number-scatter}
The result can be immediately extended to the case where we have a finite
number of scattering maps $\sigma _{1},\ldots ,\sigma _{L}$ to shadow
\begin{equation*}
x_{i+1}=f^{m_{i}}\circ \sigma _{\alpha _{i}}\left( x_{i}\right) ,\qquad
i=0,\ldots ,l-1,\,l\geq 1,
\end{equation*}%
for two prescribed sequences $m_{1},\ldots ,m_{l}\in \mathbb{N}$ and $\alpha
_{1},\ldots ,\alpha _{l}\in \left\{ 1,\ldots ,L\right\} $; see \cite[Theorem
3.7]{GLS}.
\end{remark}
\cor{
\begin{remark} In the setting of the restricted three body problem $\Lambda$ will be a normally hyperbolic cylinder with boundary, but we will embed it in a two dimensional torus. Details are given in section \ref{sec:diffusion-Hamiltonian}. We will therefore be in a setting where $\Lambda$ is a compact manifold without boundary. 
\end{remark}}

\cor{
\begin{remark}
\label{rem:symplectic} 
If $\Lambda$  has finite measure and $f$ is measure-preserving on $\Lambda$ then by the Poincar\'{e}
recurrence theorem we can take $\mathcal{U}=\Lambda$.
\end{remark}
}


\section{Diffusion mechanism}\label{sec:diff-mech}

In this section we recall the results from \cite{Jorge}. 
These are based on \cite{GLS}, but were adapted in  \cite{Jorge} to allow for computer assisted validation of the required assumptions.

\subsection{Diffusion for maps}
\cor{Let 
\[f_{\varepsilon}:\mathbb{R}^2 \times \mathbb{T} ^2 \to \mathbb{R}^2 \times \mathbb{T}^2 \]
be a family of smooth maps, which is smoothly parameterised by $\varepsilon\in \mathbb{R}$. We shall use the following notations for our coordinates: $(u,s,I,\theta)\in \mathbb{R}^2 \times \mathbb{T}^2$. The coordinates $u$ and $s$ stand for the `unstable' and `stable' coordinates
of $\Lambda_{0}$, respectively, and $T,\theta$ will play the role of `central' coordinates.} 

We assume that for $\varepsilon=0$ the coordinate $I$ is preserved by the map
\begin{equation}
\pi_{I}f_{0}\left(  z\right)  =\pi_{I}z\qquad\text{for all }z\in\tilde{\Sigma
}. \label{eq:energy-preserved}%
\end{equation}
and that
\begin{equation}
\Lambda_{0}=\left\{  \left(  u,s,I,\theta\right)  :u=s=0,  \cor{I \in \mathbb{T}, \theta \in \mathbb{T}}\right\}  , \label{eq:Lambda0-for-maps}
\end{equation}
is a normally hyperbolic invariant manifold for $f_0$. \cor{Note that $\Lambda_{0}$ is a two dimensional torus, hence it is compact without boundary. } \correction{comment 24}{}\correction{comment 26}{}

\cor{By the normally hyperbolic invariant manifold theorem $\Lambda_0$ will be perturbed to an $f_{\varepsilon}$ invariant torus $\Lambda_{\varepsilon}$, for sufficiently small $\varepsilon$. We shall assume that $f_{\varepsilon}$ are measure preserving on $\Lambda_{\varepsilon}$.  }

Let $g:\mathbb{R}\times\mathbb{R}^{3}\times\mathbb{T}\rightarrow\mathbb{R}%
^{3}\times\mathbb{T}$ be defined as\correction{comment 23}{}\correction{comment 25}{}\cor{
\begin{align*}
g\left(0,x\right) & := \frac{\partial f_{\varepsilon}}{\partial \varepsilon}(x)|_{\varepsilon=0}, \\
g\left(  \varepsilon,x\right)  & :=\frac{1}{\varepsilon}\left(  f_{\varepsilon
}(x)-f_{0}(x)\right)  \qquad \mbox{for } \varepsilon \ne 0.
\end{align*}}
Then%
\[
f_{\varepsilon}(x)=f_{0}(x)+\varepsilon g(\varepsilon,x).
\]

The following theorem provides conditions under which for any sufficiently
small $\varepsilon>0$ there exists a point $x_{\varepsilon}$ and a number of
iterates $n_{\varepsilon}$ for which
\begin{equation}
\pi_{I}\left(  f_{\varepsilon}^{n_{\varepsilon}}\left(  x_{\varepsilon
}\right)  -x_{\varepsilon}\right)  >1. \label{eq:diffusion-orbit-condition}%
\end{equation}
We first give a definition and follow with the statement of the theorem.

\begin{figure}[ptb]
\begin{center}
\includegraphics[height=2.6cm]{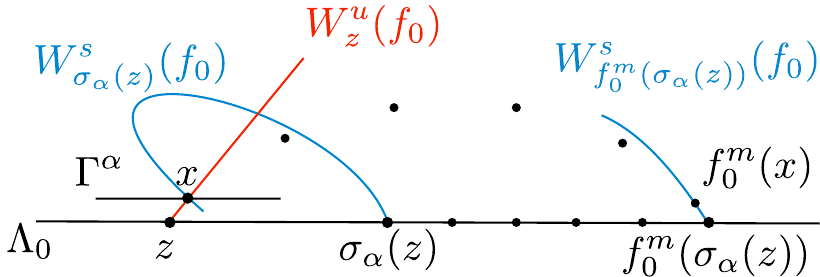}
\end{center}
\caption{The setting for Theorem \ref{th:strip-up}. \cite{Jorge}}%
\label{fig:main-thm}%
\end{figure}

\begin{definition}\label{def:strip}
Consider the topology on $\Lambda_{0}\cap\left\{  I\in\left[  0,1\right]
\right\}  $\footnote{\cor{Recall that $\mathbb{T} = \mathbb{R}/2\pi$, so $[0,1]$ is a closed interval in $\mathbb{T}$.}} induced by $\Lambda_{0}$. We say that an open set $S\subset
\Lambda_{0}\cap\left\{  I\in\left[  0,1\right]  \right\}  $ is a strip in
$\Lambda_{0}$ iff%
\[
S\cap\left\{  z\in\Lambda_{0}:\pi_{I}z=\iota\right\}  \neq\emptyset
\qquad\text{for any }\iota\in\left[  0,1\right]  .
\]

\end{definition}


\begin{theorem}
\cite{Jorge}\label{th:strip-up} Assume that there is a neighborhood $U$ of
$\Lambda_{0}$ and a positive constant $L_{g}$ such that for every $z\in\Lambda_{0},x_{u}\in W_{z}^{u}(f_{0},U),x_{s}\in W_{z}^{s}(f_{0},U),$ 
\begin{equation}
\begin{array}
[c]{c}%
\left\vert \pi_{I}\left(  g(0,x_{u})-g\left(  0,z\right)  \right)
\right\vert \leq L_{g}\left\Vert x_{u}-z\right\Vert ,\medskip\\
\left\vert \pi_{I}\left(  g(0,x_{s})-g\left(  0,z\right)  \right)
\right\vert \leq L_{g}\left\Vert x_{s}-z\right\Vert.
\end{array}
\label{eq:Lip-g-assumption}%
\end{equation}
Assume also that there exist positive constants $C,\lambda$, where $\lambda
\in(0,1)$, such that for every $z\in\Lambda_{0}$ and every $x_{u}\in W_{z}^{u}(f_{0}%
,U),x_{s}\in W_{z}^{s}(f_{0},U)$ we have
\begin{equation}%
\begin{array}
[c]{lll}%
\left\Vert f_{0}^{n}\left(  z\right)  -f_{0}^{n}\left(  x_{u}\right)
\right\Vert <C\lambda^{\left\vert n\right\vert } & \qquad & \text{for all
}n\leq0,\medskip\\
\left\Vert f_{0}^{n}\left(  z\right)  -f_{0}^{n}\left(  x_{s}\right)
\right\Vert <C\lambda^{n} &  & \text{for all }n\geq0.
\end{array}
\label{eq:contraction-expansion-bounds}%
\end{equation}
Assume also that for $\varepsilon=0$ we have the sequence of scattering maps
$\sigma_{\alpha}:\mathrm{dom}\left(  \sigma_{\alpha}\right)  \rightarrow
\Lambda_{0}$ for $\alpha=1,\ldots,L$. Let $S^{+}\subset\Lambda_{0}$ be a
strip. Assume that for every $z\in\overline{S^{+}}$

\begin{enumerate}
\item there exists an $\alpha\in\left\{  1,\ldots,L\right\}  $ for which
$z\in\mathrm{dom}\left(  \sigma_{\alpha}\right)  $, and there exists a constant $m\in\mathbb{N}$ such that
\begin{equation}
f_{0}^{m}\circ\sigma_{\alpha}\left(  z\right)  \in S^{+},
\label{eq:strip-return-cond}%
\end{equation}

\item  there exists a point $x\in W_{z}%
^{u}\left(  f_{0},U\right)  \cap W_{\sigma_{\alpha}\left(  z\right)  }%
^{s}\left(  f_{0}\right)  $ such that $f_{0}^{m}\left(  x\right)  \in
W_{f_{0}^{m}\left(  \sigma_{\alpha}\left(  z\right)  \right)  }^{s}(f_{0},U)$
and%
\begin{equation}
\sum_{j=0}^{m-1}\pi_{I}g\left(  0,f_{0}^{j}\left(  x\right)  \right)
-\frac{1+\lambda}{1-\lambda}L_{g}C>0. \label{eq:key-assumption-again}%
\end{equation}
(The choice of $m$ and $\alpha$ can depend on $z$.)
\end{enumerate}
Then for sufficiently small $\varepsilon>0$ there exists an $x_{\varepsilon}$
and $n_{\varepsilon}>0$ such that%
\begin{equation}
\pi_{I}\left(  f_{\varepsilon}^{n_{\varepsilon}}\left(  x_{\varepsilon
}\right)  -x_{\varepsilon}\right)  >1. \label{eq:drift-up-th1}
\end{equation}

\end{theorem}

The following theorem can be used to establish orbits of the perturbed system,
whose $I$ coordinate decreases.

\begin{theorem}
\cite{Jorge}\label{th:strip-down}Assume that conditions
(\ref{eq:Lip-g-assumption}) and (\ref{eq:contraction-expansion-bounds}) are
satisfied, and that for $\varepsilon=0$ we have the sequence of scattering
maps $\sigma_{\alpha}:\mathrm{dom}\left(  \sigma_{\alpha}\right)
\rightarrow\Lambda_{0}$ for $\alpha=1,\ldots,L$. Let $S^{-}\subset\Lambda_{0}$
be a strip. Assume that for every $z\in\overline{S^{-}}$

\begin{enumerate}
\item there exists an $\alpha\in\left\{  1,\ldots,L\right\}  $ for which
$z\in\mathrm{dom}\left(  \sigma_{\alpha}\right)  $, and there exists a constant $m\in\mathbb{N}$ such that 
\[
f_{0}^{m}\circ\sigma_{\alpha}\left(  z\right)  \in S^{-},
\]

\item  there exists a point $x\in W_{z}%
^{u}\left(  f_{0},U\right)  \cap W_{\sigma_{\alpha}\left(  z\right)  }%
^{s}\left(  f_{0}\right)  $ such that $f_{0}^{m}\left(  x\right)  \in
W_{f_{0}^{m}\left(  \sigma_{\alpha}\left(  z\right)  \right)  }^{s}(f_{0},U)$
and%
\[
\sum_{j=0}^{m-1}\pi_{I}g\left(  0,f_{0}^{j}\left(  x\right)  \right)
+\frac{1+\lambda}{1-\lambda}L_{g}C<0.
\]
(The choice of $m$ and $\alpha$ can depend on $z$.)
\end{enumerate}
Then for sufficiently small $\varepsilon>0$ there exists an $x_{\varepsilon}$
and $n_{\varepsilon}>0$ such that%
\begin{equation}
\pi_{I}\left(  x_{\varepsilon}-f_{\varepsilon}^{n_{\varepsilon}}\left(
x_{\varepsilon}\right)  \right)  >1. \label{eq:strip-down-th2}
\end{equation}

\end{theorem}

By combining the two strips $S^{+}$ and $S^{-}$ we obtain shadowing of any
prescribed finite sequence of actions.

\begin{theorem}
\cite{Jorge}\label{th:shadowing-seq} Assume that two strips $S^{+}$ and
$S^{-}$ satisfy assumptions of Theorems \ref{th:strip-up} and
\ref{th:strip-down}, respectively. If in addition

\begin{enumerate}
\item for every $z\in\overline{S^{+}}$ there exists an $n$ (which can depend
on $z$) such that $f_{0}^{n}(z)\in S^{-}$, and

\item for every $z\in\overline{S^{-}}$ there exists an $n$ (which can depend
on $z$) such that $f_{0}^{n}(z)\in S^{+}$,
\end{enumerate}
then there exists an $\mathcal{M}$ such that for any given finite sequence $\{I_{k}\}_{k=0}^{N}$ and any given
$\delta>0$, for sufficiently small $\varepsilon$ there exists an orbit of
$f_{\varepsilon}$ which $\varepsilon\mathcal{M}$-shadows the actions $I_{k}$; i.e. there
exists a point $z_{0}^{\varepsilon}$ and a sequence of integers $n_{1}%
^{\varepsilon}\leq n_{2}^{\varepsilon}\leq\ldots\leq n_{N}^{\varepsilon}$ such
that
\begin{equation}
\left\Vert \pi_{I}f_{\varepsilon}^{n_{k}^{\varepsilon}}\left(z_{0}^{\varepsilon
}\right)-I_{k}\right\Vert <\varepsilon\mathcal{M}.\label{eq:shadowing-seq-th3}
\end{equation}
\end{theorem}



\subsection{Diffusion for time periodic perturbations of Hamiltonian systems\label{sec:diffusion-Hamiltonian}}

Consider the following family of Hamiltonian systems $H_{\varepsilon
}:\mathbb{R}^{4}\times\mathbb{T}\rightarrow\mathbb{R}$, that depends smoothly
on the parameter $\varepsilon\in\mathbb{R}$, which generates the following ODE
in the extended phase space%
\begin{align}
x^{\prime}  &  =J\nabla_x H_{\varepsilon}\left(  x,t\right)
,\label{eq:Hamiltonian-form-of-problem}\\
t^{\prime}  &  =1,\nonumber
\end{align}
where
\[
J=\left(
\begin{array}
[c]{ll}%
0 & Id\\
-Id & 0
\end{array}
\right)  ,\qquad\text{for}\qquad Id=\left(
\begin{array}
[c]{ll}%
1 & 0\\
0 & 1
\end{array}
\right)  .
\]
Let $\Phi_{t}^{\varepsilon}$ be the flow of
(\ref{eq:Hamiltonian-form-of-problem}). We refer to $H_{\varepsilon=0}$ as the
unperturbed system. We make the following important assumption 
\[
H_{\varepsilon=0}\left(  x,t\right)  =H\left(  x\right)  ,
\]
where $H:\mathbb{R}^{4}\rightarrow\mathbb{R}$; in other words, we assume that
the unperturbed system is autonomous and hence $H$ is a constant of motion. 

Consider a \correction{comment 20}{ local} Poincar\'{e} section $\Sigma$ in $\mathbb{R}^{4}$ for  the system $x^{\prime}=J\nabla H\left(  x\right)  $, and
consider a section $\tilde{\Sigma}=\Sigma\times\mathbb{T}$ for the perturbed
system (\ref{eq:Hamiltonian-form-of-problem}) in the extended phase space. Let
us consider a Pioncar\'{e} map%
\begin{equation}
f_{\varepsilon}:\tilde{\Sigma}\rightarrow\tilde{\Sigma}.
\label{eq:maps-for-diff-mech}%
\end{equation}

\correction{comment 21}{Due to the Hamiltonian nature of the system, the map $f_{\varepsilon}$ is
symplectic for a suitable symplectic form $\omega_{\varepsilon}$ on $\tilde \Sigma$.}
\cor{
\begin{remark}
We believe that the fact that a Poincar\'e map of a Hamiltonian system in the extended phase space is symplectic, for a symplectic form induced from the standard one, is a well known result. We have been unable to find a source for this though, so we include a short section in \ref{sec:poinc-symplectic} with the proof and an explicit formula (\ref{eq:symplectic-form-in-extended-space-app}) for the choice of the symplectic form. 
\end{remark}}

The coordinates on $\tilde{\Sigma}$ can be identified with $\mathbb{R}%
^{3}\times\mathbb{T}$. We will write $\theta\in\mathbb{T}$ for the extended
phase space (time) coordinate. We assume that one of the remaining coordinates
on $\tilde{\Sigma}$ is the energy, expressed by the Hamiltonian $H$. We will
write $I$ for this coordinate. Since for $\varepsilon=0$ the energy is
preserved we have
\begin{equation*}
\pi_{I}f_{0}\left(  z\right)  =\pi_{I}z\qquad\text{for all }z\in\tilde{\Sigma
}.
\end{equation*}

We assume that for $\varepsilon=0$ the map $f_{\varepsilon=0}$ has a normally
hyperbolic invariant manifold $\Lambda_{0}\subset\tilde{\Sigma}$.
We assume that $\Lambda_{0}$ is parameterised by $\theta,I$. For the remaining
two coordinates on $\tilde{\Sigma}$ we will write $u$ and $s$. We assume that 
\[
\Lambda_{0}=\left\{  \left(  u,s,I,\theta\right)  :u=s=0, \cor{I\in \mathbb{R}, \theta \in \mathbb{T}} \right\}  ,
\]
and that $\omega_{\varepsilon=0}|_{\Lambda_{0}}$ is non-degenerate.

\correction{comment 22}{}
\cor{
The manifold $\Lambda_{0}$ is a cylinder, but we are intersected in proving that after perturbation we have changes in energy $I=H$ for $I\in [0,1]$. We will perform the following construction to embed $\Lambda_{0}\cap \{I\in [0,1]\}$ in a torus; that will allow us to assume that $\Lambda_0$ is a torus:

The idea, which we describe in more detail below, is that for $I\in\left[ 0,1\right] $ we leave the system as it is. We then employ
a `bump' function so that at the edges of the domain $I\in\left[ -1,2\right] 
$, i.e. for $I=2$ and $I=-1$, we have $\tilde{f}_{\varepsilon}=f_{0}$%
. For the remaining $I\notin (-1,2)$ we `freeze' the system taking $\tilde{f%
}_{\varepsilon}=f_{0}$. We can therefore `glue' the system by considering $I\in \mathbb{R}/2\pi$, which turns $\Lambda_0$ from a cylinder to a torus.

In more detail, we consider a smooth `bump' function\footnote{%
For instance $b(x)=\exp (-\left( 1-x^{2}\right) ^{-1})$ for $x\in \left[ -1,0%
\right] $, $b(x)=1$ for $x\in \left[ 0,1\right] $, $b(x)=\exp (-\left(
1-(1-x)^{2}\right) ^{-1})$ for $x\in \left[ 1,2\right] $ and zero otherwise.}
$b:\mathbb{R}\rightarrow \left[ 0,1\right] $ for which%
\begin{align*}
b\left( I\right) & =0\qquad \text{for }I\in \mathbb{R}\setminus \left(
-1,2\right) , \\
b\left( I\right) & =1\qquad \text{for }I\in \left[ 0,1\right] ,
\end{align*}%
and take $\tilde{f}_{\varepsilon }$ to be the Poincar\'e maps for a modified ODE 
\begin{equation*}
x^{\prime }=J\nabla _{x}\left( H\left( x\right) +b(H(x))\varepsilon
G(x,t)\right) .
\end{equation*}%
We then have the following lemma:
\begin{lemma}
\label{lem:kam-modified}\cite[Lemma 35]{Jorge} If for $\tilde{f}_{\varepsilon }|_{\Lambda
_{\varepsilon }}$we have a Cantor set of KAM tori on $\Lambda _{\varepsilon }
$, and assumptions of Theorem \ref{th:strip-up} (Theorem \ref{th:strip-down}, Theorem \ref{th:shadowing-seq}) are satisfied
for $\tilde{f}_{\varepsilon }$, then there exists an $x_{\varepsilon }$ for
which the condition (\ref{eq:drift-up-th1}) ((\ref{eq:strip-down-th2}), (\ref{eq:shadowing-seq-th3})) holds for $f_{\varepsilon }$.
\end{lemma}
Working with $\tilde f_{\varepsilon}$ allows us to treat $\Lambda_0$ as a compact manifold without boundary. (In fact  $\Lambda_0$ is a two dimensional trous.) Thus we can use the standard version of the normally hyperbolic invariant manifold theorem \cite{MR292101} and make use of Remark \ref{rem:symplectic} for the shadowing by means of Theorem \ref{th:shadowing}.
}

\subsection{The Planar Elliptic Restricted Three Body Problem}

The Planar Elliptic Restricted Three Body Problem (PER3BP) describes the
motion of a massless particle (e.g., an asteroid or a spaceship, which mass is in fact negligible in comparison to the masses of other celestial objects), under the
gravitational pull of two large bodies, which we call primaries. The primaries
rotate in a plane along Keplerian elliptical orbits with eccentricity
$\varepsilon$, while the massless particle moves in the same plane and has no
influence on the orbits of the primaries. We use normalized units, in which
the masses of the primaries are $\mu$ and $1-\mu$. We consider a frame of
`pulsating' coordinates that rotates together with the primaries, making their
position fixed on the horizontal axis \cite{S}. The motion of the massless
particle is described via the Hamiltonian $H_{\varepsilon}:\mathbb{R}%
^{4}\times\mathbb{T\rightarrow R}$
\begin{equation}%
\begin{split}
H_{\varepsilon}\left(  X,Y,P_{X},P_{Y},\theta\right)   &  =\frac{\left(
P_{X}+Y\right)  ^{2}+\left(  P_{Y}-X\right)  ^{2}}{2}-\frac{\Omega\left(
X,Y\right)  }{1+\varepsilon\cos(\theta)},\\
\Omega\left(  X,Y\right)   &  =\frac{1}{2}\left(  X^{2}+Y^{2}\right)
+\frac{\left(  1-\mu\right)  }{r_{1}}+\frac{\mu}{r_{2}},\\
r_{1}^{2}  &  =\left(  X-\mu\right)  ^{2}+Y^{2},\\
r_{2}^{2}  &  =\left(  X-\mu+1\right)  ^{2}+Y^{2}.
\end{split}
\label{eq:H-pre3bp}%
\end{equation}

The corresponding Hamilton equations are:
\begin{equation}%
\begin{array}
[c]{lll}%
\displaystyle\frac{dX}{d\theta}=\frac{\partial H_{\varepsilon}}{\partial
P_{X}},\smallskip & \qquad\qquad & \displaystyle\frac{dP_{X}}{d\theta}%
=-\frac{\partial H_{\varepsilon}}{\partial X},\\
\displaystyle\frac{dY}{d\theta}=\frac{\partial H_{\varepsilon}}{\partial
P_{Y}}, & \qquad\qquad & \displaystyle\frac{dP_{Y}}{d\theta}=-\frac{\partial
H_{\varepsilon}}{\partial Y},
\end{array}
\label{eq:PER3BP-ode}%
\end{equation}
where $X,Y\in\mathbb{R}$ are the position coordinates of the massless
particle, and $P_{X},P_{Y}\in\mathbb{R}$ are the associated linear momenta. \correction{comment 5}{We use the convention, in which in the $X,Y$ coordinates the Jupiter lies to the left of the origin at $\left(  \mu-1,0\right)  $, and the Sun is to the right of the origin at $\left(  \mu,0\right)  $.}
The variable $\theta\in\mathbb{T}$ is the true anomaly of the Keplerian orbits
of the primaries, where $\mathbb{T}$ denotes the $1$-dimensional torus. The
system is non-autonomous, thus we consider it in the extended phase space, of
dimension $5$, with $\theta$ as an independent variable. We use the
notation $\Phi_{t}^{\varepsilon}$ to denote the flow of (\ref{eq:PER3BP-ode})
in the extended phase space, which includes $\theta\in\mathbb{T}$, i.e.
\[
\Phi_{t}^{\varepsilon}:\mathbb{R}^{4}\times\mathbb{T}^{1}\rightarrow
\mathbb{R}^{4}\times\mathbb{T}^{1}.
\]

When $\varepsilon=0$ the primaries rotate around the center of mass along
circular orbits. The PER3BP becomes the Planar Circular Restricted Three Body
Problem (PCR3BP). We shall use the notation $\Phi_{t}$ for the flow of the
PCR3BP. Since the ODE of the PCR3BP is autonomous, this flow is not in the
extended phase space, i.e. for a given $t\in\mathbb{R}$,%
\[
\Phi_{t}:\mathbb{R}^{4}\rightarrow\mathbb{R}^{4}.
\]
We naturally have the relation $\Phi_{t}\left(  x\right)  =\pi_{x}\Phi_{t}%
^{0}\left(  x,\theta\right)  $ (for $\varepsilon=0$ the right hand side does
not depend on the choice of $\theta$).

In the PCR3BP the system has the following time reversing symmetry
\begin{equation}
\mathcal{R}\circ\Phi_{t}=\Phi_{-t}\circ\mathcal{R}, \label{eq:symmetry-3bp}%
\end{equation}
where%
\[
\mathcal{R}\left(  X,Y,P_{X},P_{Y}\right)  =\left(  X,-Y,-P_{X},P_{Y}\right)
.
\]
We say that an orbit $x\left(  t\right)  $ is $\mathcal{R}$-symmetric if
\[
x(-t)=\mathcal{R}\left(  x(t)\right)  .
\]
Note that for such orbits we need to have $\pi_{Y}x\left(  0\right)
=\pi_{P_{X}}x\left(  0\right)  =0$.

The Jacobi integral $-2\cdot H_{0}$ is a preserved coordinate for the
unperturbed system with $\varepsilon=0.$ We shall refer to $H_{0}$ as the
energy. For a point $x\in\mathbb{R}^{4}$ we shall write%
\[
I(x):=H_{0}\left(  x\right)  .
\]


\section{Statement of the main result\label{sec:main}}

The PCR3BP has five libration fixed points. Three of these are on $\{Y=0\}$;
we refer to them as collinear and denote as $L_{1},L_{2},L_{3}$. One of the collinear libration points, which we denote as $L_{1}$ lies between
the Sun and Jupiter. Around this fixed point we have a family of Lyapunov
periodic orbits. Each orbit has different energy. The Lyapunov orbits are
$\mathcal{R}$-symmetric and we can choose points belonging to the orbits of
the form $\left(  X,0,0,P_{X}\right)  $, with suitable choices of $X$ and
$P_{X}$. The $P_{X}$ depends on the choice of $X$, so we shall write
$P_{X}=P_{X}\left(  X\right)  $. 

The Lyapunov orbits considered by us will be parameterised by $X$. From now
on, we will write $\mathcal{X}\in\mathbb{R}$ to denote the value
$X=\mathcal{X}$, which determines the choice of a point on the Lyapunov orbit
$\left(  \mathcal{X},0,0,P_{X}\left(  \mathcal{X}\right)  \right)  $. This
will allow us to make the distinction between $\mathcal{X}$, which are values
parameterising the orbits and the coordinate $X$. We will be interested in the
family of Lyapunov orbits with%
\begin{equation}
\mathcal{X}\in\left[  -0.95,-0.95+10^{-9}\right]  .\label{eq:X-interval}%
\end{equation}

The energy of such orbits, measured by the Jacobi integral, is approximately $3.03$, which is the energy of the comet Oterma \cite{MR1765636} that was observed in the Jupiter-Sun system. We choose to work with this particular energy, since it has a physical meaning of a known celestial body, but we could easily work with a different $\mathcal{X}$-interval than (\ref{eq:X-interval}), since there is nothing special about the energy of Oterma. In fact, choosing a different energy level can make some technical aspects of the computer assisted proof easier. (We elaborate on this in Remark \ref{rem:Oterma-m}.) We fix an energy level of some physical object to show that the method can be applied in a concrete setting.

For each $\mathcal{X}$ from (\ref{eq:X-interval}) the corresponding Lyapunov
orbit has a different energy $I(\mathcal{X})$. We will show that for
sufficiently small $\varepsilon>0$, that is for the PER3BP, we can visit the
arbitrary $I(\mathcal{X})$ for $\mathcal{X}$ from (\ref{eq:X-interval}). We
shall refer to these as diffusing orbits. Formally, the main result of this
paper is as follows.

\begin{theorem}
\label{th:main}(Main theorem) Let $I(\mathcal{X})$ denote the energy of a
Lyapunov orbit starting from a given Lyapunov orbit, i.e.
\[
I\left(  \mathcal{X}\right)   =H_0\left(  \mathcal{X},0,0,P_{X}(\mathcal{X}%
)\right)  .
\]
There exists a constant $\mathcal{M}>0$ such that for an arbitrary (finite) sequence $\mathcal{X}_{1},\ldots,\mathcal{X}%
_{N}$ from the interval (\ref{eq:X-interval}) and for sufficiently small $\varepsilon>0$, there exists a
sequence  $t_{1}^{\varepsilon},\ldots,t_{N}^{\varepsilon}$ and a point
$x^{\varepsilon}$ \correction{comment 16}{ }such that%
\[
\left\Vert H_0\left(  \Phi_{t_{i}^{\varepsilon}}^{\varepsilon}(x^{\varepsilon
})\right)  -I\left(  \mathcal{X}_{i}\right)  \right\Vert <\varepsilon\mathcal{M}\qquad
\qquad\text{for }i=1,\ldots,N.
\]

\end{theorem}
We now make a few remarks concerning the result.

\correction{comment 3}{}\correction{comment 13}{}\cor{We obtain our diffusion result for sufficiently small $\varepsilon$. We believe that in some settings it is possible to extend the method from section \ref{sec:diff-mech} to consider explicit sizes of perturbations. The first difficulty in such extension would be that that these methods rely on the persistence of normally hyperbolic manifolds. One would therefore need to establish the existence of such manifolds for explicit perturbations. This is doable 
\cite{MR3467671,MR3309008,MR3397322}, but not straightforward. Additionally, in this paper we deal with a normally hyperbolic cylinder with boundary and hence we rely on the existence of KAM tori after perturbation. KAM theorem ensures that the inner dynamics on the perturbed manifold is contained between the invariant tori, which allows us to shadow arbitrary sequences of energy changes $\mathcal{X}_{1},\ldots,\mathcal{X}_{N}$ by using the outer dynamics. It is possible to obtain KAM results for explicit perturbations 
\cite{CALLEJA2022106099}, but this is far from straightforward. Without KAM the best that we could prove is the existence of orbits which have an explicit energy change. Such diffusion would follow from a dichotomy that there can be either the diffusion without leaving the manifold or the diffusion through the outer dynamics excursions.
}

The interval (\ref{eq:X-interval}) is very narrow. The actual physical
distance between the two points on the Lyapunov orbits on the section $\{Y=0\}$,
which correspond to endpoints of (\ref{eq:X-interval}) is about $1$\thinspace km.
This means that the diffusion established by us is over a very narrow range,
but it is not completely negligible.

The proof was performed with computer-assisted tools, and for the interval
(\ref{eq:X-interval}) it took 17 minutes, running on a single thread on a
standard laptop. Our computer assisted proof could be performed on other
$\mathcal{X}$-intervals, which combined together would lead to diffusion over
longer distances. Such proof could be performed by parallel computations on a
cluster. We are more interested in the proof of concept rather than in
obtaining long intervals by brute force, so we have not performed such validation. 


\section{Proof of the main result\label{sec:proof}}

The proof is performed in the following steps:

\begin{enumerate}
\item \label{pt:Lyap}establish the existence of the family of Lyapunov orbits,
which will form the NHIM $\Lambda_{0}$,

\item \label{pt:local-unstable}establish bounds for the local stable/unstable
manifolds of $\Lambda_{0}$,

\item \label{pt:inter}prove that the stable and unstable manifolds of
$\Lambda_{0}$ intersect transversally, and that we have a homoclinic channel
along the intersection,

\item \label{pt:persistence}prove that $\Lambda_{0}$ is perturbed to $\Lambda_{\varepsilon}$, which contains a Cantor set of KAM tori. 

\item \label{pt:main-proof}apply Theorem \ref{th:shadowing-seq} to obtain the
existence of diffusing orbits.
\end{enumerate}

Section \ref{sec:shooting} sets out a parallel shooting method, which is then
used for the steps \ref{pt:Lyap} and \ref{pt:inter} in sections \ref{sec:Lyap}
and \ref{sec:inter}, respectively. Sections \ref{sec:local-unstable},
\ref{sec:persistence} and \ref{sec:main-proof} deal with steps
\ref{pt:local-unstable}, \ref{pt:persistence} and \ref{pt:main-proof}, respectively.


\subsection{Parallel shooting for symmetric orbits\label{sec:shooting}}

Here we consider the PCR3BP and work with the flow $\Phi_{t}$ in
$\mathbb{R}^{4}$. Let us consider a $C^{1}$ function $p:\mathbb{R}%
\rightarrow\mathbb{R}^{4}$. For now we leave $p$ unspecified. Later on,
depending on the choice of $p$ we will use the below method to establish
either bounds on the points along a Lyapunov orbit, or for a homoclinic orbit
to a Lyapunov orbit.

Let us define the following function%
\[
F:\mathbb{R}\times\mathbb{R}\times\underset{n}{\underbrace{\mathbb{R}%
^{4}\times\ldots\times\mathbb{R}^{4}}}\rightarrow\underset{n}{\underbrace
{\mathbb{R}^{4}\times\ldots\times\mathbb{R}^{4}}}\times\mathbb{R}%
\times\mathbb{R},
\]
as%
\begin{align}
&  F\left(  s,\tau,x_{1},\ldots,x_{n}\right)  :=\label{eq:F-for-shooting}\\
&  \left(  \Phi_{\tau}\left(  p(s)\right)  -x_{1},\Phi_{\tau}\left(
x_{1}\right)  -x_{2},\ldots,\Phi_{\tau}\left(  x_{n-1}\right)  -x_{n},\pi
_{Y}\Phi_{\tau}\left(  x_{n}\right)  ,\pi_{P_{X}}\Phi_{\tau}\left(
x_{n}\right)  \right)  .\nonumber
\end{align}
We see that if we find a point $\mathbf{x}^{\ast}=\left(  s,\tau,x_{1}%
,\ldots,x_{n}\right)  $ such that
\begin{equation}
F\left(  \mathbf{x}^{\ast}\right)  =0, \label{eq:F-zeor-R-sym}%
\end{equation}
then by taking $x_{0}=p(s)$ and $x_{n+1}=\Phi_{\tau}\left(  x_{n}\right)  $ we
obtain%
\[
\Phi_{\tau}\left(  x_{i}\right)  =x_{i+1}\qquad\text{for }i=0,\ldots,n.
\]
Moreover, since
\[
\pi_{Y}x_{n+1}=\pi_{Y}\Phi_{\tau}\left(  x_{n}\right)  =0,\qquad\qquad
\pi_{P_{X}}x_{n+1}=\pi_{P_{X}}\Phi_{\tau}\left(  x_{n}\right)  =0,
\]
we see that $x_{n+1}$ is $\mathcal{R}$-symmetric, i.e. $x_{n+1}=\mathcal{R}%
\left(  x_{n+1}\right)  .$ If we now define $x_{n+k}=\mathcal{R}\left(
x_{n+2-k}\right)  $, for $k=2,\ldots,n+2$ then by (\ref{eq:symmetry-3bp}) we
will obtain an $\mathcal{R}$-symmetric orbit $x(t)$, which starts from
$x\left(  0\right)  =x_{0}$, and for which%
\[
\Phi_{\tau}\left(  x_{i}\right)  =x_{i+1}\qquad\text{for }i=0,\ldots,2n+1.
\]

Solving of (\ref{eq:F-zeor-R-sym}) can be done by means of the Krawczyk
method. The first step is to obtain an approximate solution of
(\ref{eq:F-zeor-R-sym}) by iterating (using non-rigorous numerics)%
\begin{equation}
\mathbf{x}_{i+1}=\mathbf{x}_{i}-\left(  DF\left(  \mathbf{x}_{i}\right)
\right)  ^{-1}F\left(  \mathbf{x}_{i}\right)  .\label{eq:Newton-initial}%
\end{equation}
After a few iterates of (\ref{eq:Newton-initial}) we obtain a point
$\mathbf{x}$ around which we can construct a cube $\mathbf{X}$ and validate
that we have the solution of (\ref{eq:F-zeor-R-sym}) inside of $\mathbf{X}$ by
using Theorem \ref{th:Krawczyk}. (We take $C$ as the non-rigorously computed inverse of the derivative $C=\left(DF\left(  \mathbf{x}\right)\right)^{-1}$.)



\subsection{Bounds for Lyapunov orbits\label{sec:Lyap}}

When we fix some $\mathcal{X}\in\mathbb{R}$ and choose the function
$p:\mathbb{R}\rightarrow\mathbb{R}^{4}$ as
\begin{equation}
p\left(  s\right)  =\left(  \mathcal{X},0,0,s\right)  ,\label{eq:ps-for-Lyap}%
\end{equation}
then the methodology form section \ref{sec:shooting} can be used to obtain a
sequence of points $x_{0},\ldots,x_{2n+2}$ along an $\mathcal{R}$-symmetric
periodic orbit. This is because by the definition of $F$%
\[
\Phi_{(n+1)\tau}\left(  x_{0}\right)  =x_{n+1},
\]
and since $x_{0}$ and $x_{n+1}$ are self $\mathcal{R}$-symmetric, by
(\ref{eq:symmetry-3bp})
\[
\Phi_{(n+1)\tau}\left(  x_{n+1}\right)  =\Phi_{(n+1)\tau}\circ\mathcal{R}\left(
x_{n+1}\right)  =\mathcal{R}\circ\Phi_{-(n+1)\tau}\left(  x_{n+1}\right)
=\mathcal{R}\left(  x_{0}\right)  =x_{0}.
\]
We thus see that $x_{0}$ is a point on an $\mathcal{R}$-symmetric periodic
orbit with period $T=2\left(  n+1\right)  \tau$.

An advantage of the method is that we can obtain a bound for a whole family of
Lyapunov orbits. This can be done by considering an interval instead of a
single $\mathcal{X}$, for our validation. The computer assisted computation
allows us to obtain a bound on $F(x)$ and $DF\left(  x\right)  $ for
(\ref{eq:F-for-shooting}). This way we obtain a bound for points which lie on
Lyapunov orbits for an interval of values $\mathcal{X}$. In our case we take
the interval (\ref{eq:X-interval}) and use this method to validate the
following result:%

\begin{table}
\begin{center}
{\scriptsize 
\begin{tabular}{l*{4}{l}}
\normalsize{$n$} & \normalsize{$X$} & \normalsize{$Y$} & \normalsize{$P_X$} & \normalsize{$P_Y$} \\
\hline \\
0 & -0.9499999995 & 0 & 0 & -0.84134724633 \\
1 & -0.95011002908 & 0.010872319337 & -0.012750492306 & -0.84566628682 \\
2 & -0.95027977734 & 0.020942127841 & -0.021798848297 & -0.85701977629 \\
3 & -0.95016249921 & 0.02964511208 & -0.02595601236 & -0.87222441368 \\
4 & -0.94945269037 & 0.036681290981 & -0.026117965229 & -0.88881047965 \\
5 & -0.94799596417 & 0.041912835694 & -0.023738329652 & -0.90554260324 \\
6 & -0.94578584443 & 0.045275924414 & -0.020027448005 & -0.92194698888 \\
7 & -0.94292354313 & 0.046741951127 & -0.015815875271 & -0.93784624844 \\
8 & -0.93958019632 & 0.046310536541 & -0.011642040196 & -0.9531124603 \\
9 & -0.93596962345 & 0.044016227704 & -0.0078569213076 & -0.96756323289 \\
10 & -0.93232909834 & 0.039939546482 & -0.0046935571959 & -0.98092506113 \\
11 & -0.92890403848 & 0.034218254755 & -0.0022984843476 & -0.99282861416 \\
12 & -0.92593321085 & 0.027056429274 & -0.00073280826007 & -1.0028268728 \\
13 & -0.92363228182 & 0.018728871729 & 4.5125264188e-05 & -1.0104387689 \\
14 & -0.92217533629 & 0.0095779013399 & 0.00019712976878 & -1.0152203731 \\
15 & -0.92167641746 & 0 & 0 & -1.0168530766 \\
\end{tabular}
}\end{center}
\caption{Midpoints of our enclosure of points along the family of Lyapunov orbits, computed using (\ref{eq:F-zeor-R-sym}) with $n=14$.\label{tabl:Lyap}}
\end{table}

\begin{lemma}
\label{lem:Lyap-enclosure}Let $r:=3.633\cdot10^{-9}$. For every $\mathcal{X}$
from the interval (\ref{eq:X-interval}) there exists an
\[
s(\mathcal{X})\in-0.84134724633+\left[  -r,r\right]
\]
such that
\begin{equation}
x_{0}(\mathcal{X})=\left(  \mathcal{X},0,0,s(\mathcal{X})\right)
\label{eq:x0-formula}%
\end{equation}
is a point, which lies on a Lyapunov orbit, which we denote as $L_{\mathcal{X}}$. Moreover, we have a sequence of
points along $L_{\mathcal{X}}$, which passes within the $r$ distance (in
maximum norm) of the points from Table \ref{tabl:Lyap}. Moreover, we have the
following bound $T(\mathcal{X})$ for the period of $L_{\mathcal{X}}$
\begin{equation}
T(\mathcal{X})\in\lbrack3.0417517493,3.0417517846]\label{eq:Lyap-period}%
\end{equation}

\end{lemma}

\begin{remark}
In Table \ref{tabl:Lyap} we write out half of the points along the family of
periodic orbits, since the second half (the remaining fourteen points, to be
precise) follows from the $\mathcal{R}$-symmetry.
\end{remark}

\begin{remark}
Our estimate on the distance from the points from Table \ref{tabl:Lyap}
obtained by our computer program varies from point to point and is frequently
more accurate than stated in Lemma \ref{lem:Lyap-enclosure}, where we simply
write out the single $r$, which is the upper bound that can be applied to all
the points.
\end{remark}


\subsection{Bounds on the unstable manifolds of Lyapunov orbits
\label{sec:local-unstable}}

In the previous section we have shown how to compute the bound on the family
of Lyapunov orbits $L_{\mathcal{X}}$ containing points $x_{0}(\mathcal{X})$ given by
(\ref{eq:x0-formula}), with $\mathcal{X}$ from the interval
(\ref{eq:X-interval}). Here we fix a single periodic orbit for some
$\mathcal{X}$ from (\ref{eq:X-interval}) and discuss how one can obtain a
computer validated enclosure of its local unstable manifold. Before we
give the method, we introduce some notation.

We consider a Poincar\'{e} section $\Sigma=\left\{  Y=0\right\}  $ and define
$\rho:\mathbb{R}^{4}\rightarrow\mathbb{R}$ and $\mathcal{P}:\mathbb{R}%
^{4}\rightarrow\mathbb{R}^{4}$ as%
\begin{align*}
\rho\left(  x\right)   &  =\inf\left\{  t>0:\Phi_{t}\left(  x\right)
\in\Sigma\right\}  ,\\
\mathcal{P}\left(  x\right)   &  =\Phi_{\rho\left(  x\right)  }\left(
x\right)  .
\end{align*}
In other words, $\rho$ is the time along the flow to the section, and
$\mathcal{P}$ is the map which goes to the section along the flow.

\begin{remark}
The $\rho$ and $\mathcal{P}$ do not need to be globally defined. Whenever we
will use these functions in our computer assisted proofs, the CAPD library
verifies that the considered sets lie within the domains of these maps, and
that they are properly defined throughout the performed validations. (If a set
would not belong to the domain, the program would return an error and terminate.)
\end{remark}

Let us fix $\mathcal{X}$ from (\ref{eq:X-interval}) and a Lyapunov orbit $L_{\mathcal{X}}$
containing $x_{0}=x_{0}(\mathcal{X})$ given by (\ref{eq:x0-formula}). Denote the period of this orbit as $T=T(\mathcal{X})$. For convenience, let us write $N=2n+2$,
so that for our points $x_{0},\ldots,x_{N}$ on the Lyapunov orbit, whose
bounds we established in Lemma \ref{lem:Lyap-enclosure} we have
\begin{align}
\Phi_{T/N}\left(  x_{i}\right)   &  =x_{i+1}\qquad\text{for }i=0,\ldots
,N-1,\label{eq:flow-periodic-pt-to-pt}\\
x_{N}  &  =x_{0}.\nonumber
\end{align}
Note that from (\ref{eq:x0-formula}) we see that $x_{0}\in\Sigma$.

Let us consider now a sequence of invertible matrices \correction{comment 27}{}$A_{i}\in\mathbb{R}%
^{4\times4}$ for $i=0,\ldots,N$, with $A_{0}=A_{N}$ and define the following
maps%
\[
f_{i}:\mathbb{R}^{4}\rightarrow\mathbb{R}^{4},\qquad\text{for }i=1,\ldots,N,
\]
as%
\begin{align}
f_{i}\left(  v\right)   &  :=A_{i}^{-1}\left(  \Phi_{T/N}\left(
A_{i-1}v+x_{i-1}\right)  -x_{i}\right)  ,\qquad\text{for }i=1,\ldots
,N-1,\nonumber\\
f_{N}\left(  v\right)   &  :=A_{N}^{-1}\left(  \mathcal{P}\left(
A_{N-1}v+x_{N-1}\right)  -x_{N}\right) = A_{0}^{-1}\left(  \mathcal{P}\left(
A_{N-1}v+x_{N-1}\right)  -x_{0}\right)   . \label{eq:fi-def}%
\end{align}
In other words, for $i=1,\ldots,N-1$ we consider time shift maps along the
flow, expressed in local coordinates around $x_{i}$. The last map, $f_{N}$,
maps to the section $\Sigma=\left\{  Y=0\right\}  \subset\mathbb{R}^{4}$. This
means that $f_{N}\circ\ldots\circ f_{1}|_{\Sigma}:\Sigma\rightarrow\Sigma$.
From (\ref{eq:flow-periodic-pt-to-pt}) it follows that%
\[
f_{i}\left(  0\right)  =0\qquad\text{for }i=1,\ldots,N.
\]

For $F:\mathbb{R}^{4}\rightarrow\mathbb{R}^{4}$ defined as%
\begin{equation}
F=f_{N}\circ\ldots\circ f_{1} \label{eq:F-map-for-wu}%
\end{equation}
we see that the origin is a fixed point. Our objective will be to establish
bounds on the unstable manifold of the origin. To be more precise, we shall
establish bounds on the intersection of the unstable manifold of the Lyapunov
orbit with $\Sigma$, which is the unstable manifold of the origin for the map
$F|_{\Sigma}$.

First we introduce some \correction{comment 6}{notions}.

\begin{definition}
Let $\left\Vert \cdot\right\Vert $ be some norm in $\mathbb{R}^{3}.$ Let
$Q:\mathbb{R}^{4}\rightarrow\mathbb{R}$ be the function%
\begin{equation}
Q\left(  v_{1},\ldots,v_{4}\right)  =\left\vert v_{1}\right\vert -\left\Vert
\left(  v_{2},v_{3},v_{4}\right)  \right\Vert . \label{eq:cone-def-form}%
\end{equation}
We define the cone centered at a point $v\in\mathbb{R}^{4}$ as%
\[
Q^{+}\left(  v\right)  :=\left\{  w\in \mathbb{R}^{4} :Q\left(  w-v\right)  \geq0\right\}  .
\]
\end{definition}

We consider a sequence of cones defined by $Q_{i}:\mathbb{R}^{4}%
\rightarrow\mathbb{R}$, for $i=0,\ldots,N$ and assume that
\begin{equation}
Q_{N}=Q_{0}.\nonumber
\end{equation}
We take the norms $\left\Vert \cdot\right\Vert _{i}$ for $Q_{i}$ from
(\ref{eq:cone-def-form}) as
\[
\left\Vert \left(  x_{1},x_{2},x_{3}\right)  \right\Vert _{i}=\max\left\{
\left\vert x_{1}\right\vert /a_{i,1},\ldots,\left\vert x_{3}\right\vert
/a_{i,3}\right\}
\]
where $a_{i,k}\in\left(  0,1\right)  $ are fixed coefficients for
$i=1,\ldots,N$ and $k=1,2,3$. In other words, we use different norms in
(\ref{eq:cone-def-form}) to define different cones. Note that $\left\vert
y\right\vert \geq\left\Vert \left(  x_{1},x_{2},x_{3}\right)  \right\Vert
_{i}$ is equivalent to $a_{i,k}\left\vert y\right\vert \geq\left\vert
x_{k}\right\vert $, for $k=1,2,3$, so the cone $Q_{i}^{+}\left(  v\right)  $
can be expressed as
\begin{equation}
Q_{i}^{+}\left(  v\right)  =\left\{  v+\left(  t,tx_{1},tx_{2},tx_{3}\right)
:x_{k}\in\left[  -a_{i,k},a_{i,k}\right]  \text{ for }k=1,2,3\text{ and }%
t\in\mathbb{R}\right\}  . \label{eq:Qi-form}%
\end{equation}

\begin{remark}
\label{rem:cone-rep}The form (\ref{eq:Qi-form}) is convenient, since \correction{comment 7}{then the }cones defined by $Q_{i}^{+}$ can be represented in a computer assisted
implementation by a set
\[
V_{i}=\left[  1\right]  \times\left[  -a_{i,1},a_{i,1}\right]  \times
\ldots\times\left[  -a_{i,3},a_{i,3}\right]  .
\]
What we mean by this is that%
\begin{equation}
Q_{i}^{+}\left(  v\right)  =\left\{  v+tw:w\in V_{i}\text{ and }t\in
\mathbb{R}\right\}  . \label{eq:cone-set-rep}%
\end{equation}

\end{remark}

\begin{definition}
Let $B\subset\mathbb{R}^{4}$. We say that $f_{i}$ satisfies cone conditions in
$B$ iff for every $v\in B$%
\begin{equation}
f_{i}\left(  Q_{i}^{+}\left(  v\right)  \cap B\right)  \subset Q_{i+1}%
^{+}\left(  f_{i}(v)\right)  . \label{eq:cone-cond}%
\end{equation}

\end{definition}

Below lemma is our main tool for establishing bounds on the unstable manifold
of the origin for the map (\ref{eq:F-map-for-wu}).

\begin{lemma}
\cite[Lemma 6.3]{MR3032848}\label{lem:wu-cones} Let $B:=\left[  -1,1\right]
^{4}\subset\mathbb{R}^{4}$. Assume that $f_{1},\ldots,f_{N}$ satisfy cone
conditions in $B$. Let $m>1$ and assume that for $F=f_{N}\circ\ldots\circ
f_{1}$ the matrix $DF\left(  0\right)  $ has a single eigenvalue $\lambda$
satisfying $\left\vert \operatorname{Re}\lambda\right\vert >m$ and the
absolute values of the real parts of the remaining eigenvalues below $m$. If
also for every $v\in Q_{i}^{+}\left(  0\right)  \cap B$,%
\begin{equation}
\left\Vert f_{i}\left(  v\right)  \right\Vert >m\left\Vert v\right\Vert ,
\label{eq:expansion-in-cones}%
\end{equation}
then the unstable manifold of the origin for the map $F$ is parameterised as a
smooth curve $p^{\mathrm{u}}:\left[  -1,1\right]  \rightarrow B$ which
satisfies
\begin{align*}
p^{\mathrm{u}}\left(  0\right)   &  =0,\\
\pi_{1}p^{\mathrm{u}}  &  =Id,\\
p^{\mathrm{u}}\left(  \left[  -1,1\right]  \right)   &  \subset Q_{0}%
^{+}\left(  p^{\mathrm{u}}\left(  u\right)  \right)  ,\qquad\text{for every
}u\in\left[  -1,1\right]  ,
\end{align*}
and%
\[
\frac{d}{du}p^{\mathrm{u}}\left(  u\right)  \in Q_{0}^{+}\left(  0\right)
,\qquad\text{for every }u\in\left[  -1,1\right]  .
\]

\end{lemma}

\begin{remark}
In \cite[Lemma 6.3]{MR3032848} is stated for the setting where $F$ is a full
turn along the periodic orbit. Here we have a sequence of local maps, which
when composed constitute the full turn. The proof in such setting is analogous
to \cite{MR3032848}, by using the graph transform method. The only needed
modification with respect to \cite{MR3032848} is that here the graphs need to
be propagated successively by $f_{1},f_{2},\ldots,f_{N},$ after which they
return to the same local coordinates. The graph transform for each successive
$f_{i}$ follows from an identical construction as the one from
\cite{MR3032848} (which there is done for the single $F$). By composing the
graph transforms for the successive $f_{i}$ we obtain a graph transform for
$F$.
\end{remark}

\begin{remark}
The benefit of considering several maps and shooting is that this requires
shorter integration times, which improves accuracy in the interval arithmetic
computations. This is the only reason why we shoot between 29 points along
Lyapunov orbits, instead of considering a single turn.
\end{remark}

\begin{remark}
In our computer assisted proof, we choose $A_{0}$ so that the derivative of
$F=f_{N}\circ\ldots\circ f_{1}$ at the origin is close to diagonal. The
eigenvalues of $F$ are $\lambda,\frac{1}{\lambda},1,0$. (The zero comes from
the fact that $f_{N}$ maps to the section $\Sigma$.) We can validate the bound
on $\lambda$ by using the Gersgorin theorem. We choose the remaining $A_{i}$
so that the derivatives of $f_{i}$ at the origin are close to diagonal.
\end{remark}

\begin{remark}
\label{rem:cone-propagation} The validation of cone conditions is done as
follows. Take an interval set $V_{i}$ for which we have (\ref{eq:cone-set-rep}%
). Then by the mean value theorem the fact that
\begin{equation}
\left[  Df_{i}\left(  B\right)  \right]  V_{i}\subset Q_{i+1}\left(  0\right)
\label{eq:cone-valid}%
\end{equation}
implies (\ref{eq:cone-cond}). To check (\ref{eq:cone-valid}) it is enough to
compute the interval set $w=\left[  Df_{i}\left(  B\right)  \right]  V_{i}$
and validate that%
\[
\left[  \frac{w}{\pi_{1}w}\right]  \subset V_{i+1}.
\]
This means that the cone condition is straightforward to validate from the
interval enclosure of the derivative of the map. \correction{comment 8}{The CAPD library has built in methods for the computation of Poincar\'e maps and their derivatives, hence it is a good tool for the validation of the cone conditions. The methodology outlining these interval arithmetic techniques is described in \cite{PoincWZ}.}
\end{remark}

In our computer assisted proof we use the following lemma to validate
(\ref{eq:expansion-in-cones}) for the maximum norm.

\begin{lemma}
\label{lem:expansion-technical}Consider $Q\left(  v_{1},\ldots,v_{4}\right)
=\left\vert v_{1}\right\vert -\left\Vert \left(  v_{2},v_{3},v_{4}\right)
\right\Vert $ with%
\[
\left\Vert \left(  x_{1},x_{2},x_{3}\right)  \right\Vert =\max\left\{
\left\vert x_{1}\right\vert /a_{1},\left\vert x_{1}\right\vert /a_{2}%
,\left\vert x_{3}\right\vert /a_{3}\right\}  ,
\]
where $a_{1},a_{2},a_{3}\in\left(  0,1\right)  $. Take $f:B\rightarrow
\mathbb{R}^{4}$ such that $f\left(  0\right)  =0$ and let
\begin{equation}
\left[  Df\left(  B\right)  \right]  =\left(
\begin{array}
[c]{cc}%
A_{11} & A_{12}\\
A_{21} & A_{22}%
\end{array}
\right)  ,\label{eq:matrix-A-tech-1}
\end{equation}
where $A_{11}$, $A_{12}$, $A_{21}$ and $A_{22}$ are $1\times1$, $1\times3$,
$3\times1$ and $3\times3$ interval matrices, respectively. If $A_{11}>c>0$ and
$m=c-\left\Vert A_{12}\right\Vert _{\max}\max\left(  a_{1},a_{2},a_{3}\right)
$, then
\[
\left\Vert f\left(  v\right)  \right\Vert _{\max}\geq m\left\Vert v\right\Vert
_{\max}\qquad\text{for every }v\in Q^{+}\left(  0\right)  \cap B.
\]

\end{lemma}
\begin{proof}
The proof is given in \ref{sec:expansion-technical}.
\end{proof}

A mirror result can be formulated to obtain a bound in the other direction.
This is relevant for the method for obtaining the bound
(\ref{eq:Lip-g-assumption}) outlined in \ref{sec:appendix}. We place this
lemma here since it follows from similar arguments to Lemma \ref{lem:wu-cones}.

\begin{lemma}
\label{lem:Lip-technical}Consider $Q\left(  v_{1},\ldots,v_{4}\right)
=\left\vert v_{1}\right\vert -\left\Vert \left(  v_{2},v_{3},v_{4}\right)
\right\Vert $ with%
\[
\left\Vert \left(  x_{1},x_{2},x_{3}\right)  \right\Vert =\max\left\{
\left\vert x_{1}\right\vert /a_{1},\left\vert x_{1}\right\vert /a_{2}%
,\left\vert x_{3}\right\vert /a_{3}\right\}  ,
\]
where $a_{1},a_{2},a_{3}\in\left(  0,1\right)  $. Take $f:B\rightarrow
\mathbb{R}^{4}$ and let
\begin{equation}
\left[  Df\left(  B\right)  \right]  =\left(
\begin{array}
[c]{cc}%
A_{11} & A_{12}\\
A_{21} & A_{22}%
\end{array} 
\right)  , \label{eq:matrix-A-tech}
\end{equation}
where $A_{11}$, $A_{12}$, $A_{21}$ and $A_{22}$ are $1\times1$, $1\times3$,
$3\times1$ and $3\times3$ interval matrices, respectively. If%
\begin{align}
a &  :=\max\left(  a_{1},a_{2},a_{3}\right)  ,\label{eq:a-cone-upper-bound}\\
\bar{m} &  :=\max\left(  \left\vert A_{11}\right\vert +a\left\Vert
A_{12}\right\Vert _{\max},\left\Vert A_{21}\right\Vert _{\max}+a\left\Vert
A_{22}\right\Vert _{\max}\right)  ,\label{eq:m-cone-upper-bound}%
\end{align}
then%
\[
\left\Vert f\left(  v\right)  \right\Vert _{\max}\leq\bar{m}\left\Vert
v\right\Vert _{\max}\qquad\text{for every }v\in Q^{+}\left(  0\right)  \cap B.
\]

\end{lemma}
\begin{proof}The proof is given in \ref{sec:Lip-technical}.\end{proof}

\begin{remark}
In practice, the term $A_{11}$ in (\ref{eq:matrix-A-tech-1})--(\ref{eq:matrix-A-tech}), is associated with hyperbolic expansion, and
dominates in (\ref{eq:m-cone-upper-bound}). Also, in our application the number $a$ from (\ref{eq:a-cone-upper-bound}) is 
small, so we obtain $m\approx \bar{m}\approx\left\vert A_{11}\right\vert $; naturally $m< \bar m$.
\end{remark}

With the aid of Lemma \ref{lem:wu-cones} we have validated the following result.

\begin{lemma}
\label{lem:Wu-bound-CAP} Let $r=3\cdot10^{-8}$, $B=[-r,r]^{4}$, and let
\begin{equation}
A_{0}=\,\left(
\begin{array}
[c]{cccc}%
0.280324 & -0.220733 & 0.280324 & 0\\
0 & 0 & 0 & 0.816632\\
1 & 0 & -1 & -1\\
-0.343269 & 1 & -0.343269 & 0
\end{array}
\right)  .\label{eq:A0-choice}%
\end{equation}

Consider $L=5\cdot10^{-6}$ and cones defined as
\[
Q_{0}^{+}\left(  v\right)  =\left\{  v+\left(  t,tx_{2},tx_{3},tx_{4}\right)
:x_{k}\in\left[  -L,L\right]  \text{ for }k=2,3,4\text{ and }t\in
\mathbb{R}\right\}  .
\]
Then the unstable manifold of the origin for the map $F$ is parameterised as a
smooth curve $p^{\mathrm{u}}:\left[  -r,r\right]  \rightarrow B$ which
satisfies
\begin{align*}
p^{\mathrm{u}}\left(  0\right)   &  =0,\\
\pi_{1}p^{\mathrm{u}} &  =Id,\\
p^{\mathrm{u}}\left(  \left[  -r,r\right]  \right)   &  \subset Q_{0}%
^{+}\left(  p^{\mathrm{u}}\left(  u\right)  \right)  ,\qquad\text{for every
}u\in\left[  -r,r\right]  ,
\end{align*}
and%
\[
\frac{d}{du}p^{\mathrm{u}}\left(  u\right)  \in Q_{0}^{+}\left(  0\right)
,\qquad\text{for every }u\in\left[  -r,r\right]  .
\]

\end{lemma}

\begin{corollary}
\label{cor:ps-der}The curve $w^{\mathrm{u}}\left(  u\right)  :=x_{0}%
+A_{0}p^{\mathrm{u}}\left(  u\right)  $ lies on the unstable manifold of the
Lyapunov orbit containing $x_{0}.$ The curve $w^{\mathrm{u}}\left(  u\right)
$ is contained in $x_{0}+A_{0}Q_{0}^{+}\left(  0\right)  $ and $\frac{d}%
{du}w^{\mathrm{u}}\left(  u\right)  \in A_{0}V_{0}$ where%
\[
V_{0}=\left\{  1\right\}  \times\left[  -L,L\right]  \times\left[
-L,L\right]  \times\left[  -L,L\right]  .
\]
Since to define $F$ in (\ref{eq:F-map-for-wu}) we take $f_{N}=\mathcal{P}$
(see (\ref{eq:fi-def})), and $\mathcal{P}$ maps to $\Sigma$, we know that
$w^{\mathrm{u}}\left(  u\right)  \subset\Sigma$.
\end{corollary}

\begin{remark}
The matrices $A_i$ used for the coordinate changes as well as the cones $Q_i$ are chosen automatically by our computer program. We do not write them out here, since the important bounds for the proof are at the point $x_{0}=x_{0}(\mathcal{X})$, and are given in Lemma \ref{lem:Wu-bound-CAP}.
\end{remark}

\begin{remark}
To validate assumption (\ref{eq:expansion-in-cones}) of Lemma
\ref{lem:wu-cones} we have used Lemmas \ref{lem:expansion-technical} and
obtaining
$
m=1.2767546773,
$
for the bound from (\ref{eq:expansion-in-cones}).
\end{remark}

\begin{remark}
\label{rem:Lip-local}We have also obtained the following bound by using Lemma
\ref{lem:Lip-technical} for maps (\ref{eq:fi-def})%
\[
\left\Vert f_{i}\left(  v_{i}\right)  \right\Vert _{\max}\leq\bar{m}\left\Vert
v_{i}\right\Vert _{\max}\qquad\text{for }v_{i}\in Q_{i}^{+}\left(  0\right)
\cap B\text{ and }i=1,\ldots,N,
\]
where $\bar{m}=1.2767636743.$
\end{remark}

So far we have discussed how to obtain a bound on
the fiber for a fixed point at the origin of our section-to-section map $F$
defined in (\ref{eq:F-map-for-wu}). Such fixed point resulted from
intersecting the Lyapunov orbit with $\{Y=0\}$. In the future discussion we
will need to consider the problem in the extended phase space, since the
PER3BP is not autonomous. In the extended phase space the fixed point at the
origin for $F$ becomes an invariant curve $\left\{  0\right\}  \times
\mathbb{T}$. The unstable fiber for a given point $\left(  0,\lambda\right)
$, with $\lambda\in\mathbb{T}$, on this curve is contained in the extended
phase space. Since the return times to the section $\{Y=0\}$ differ from point
to point, such fiber does not need to be contained in $\left\{  \theta
=\lambda\right\}  $. Below we discuss a method with which we establish bounds
on the unstable fibers of points from $\left\{  0\right\}  \times\mathbb{T}$
in the extended phase space.

To make the above discussion more precise, we consider
the flow of the PCR3BP in the extended phase space and denote it as
$\tilde{\Phi}_{t}$. We consider the Poincar\'{e} section $\tilde{\Sigma
}=\left\{  Y=0\right\}  =\Sigma\times\mathbb{T}$, define $\tilde{\rho
}:\mathbb{R}^{4}\times\mathbb{T}\rightarrow\mathbb{R}$ and define
$\mathcal{\tilde{P}}:\mathbb{R}^{4}\times\mathbb{T}\rightarrow\mathbb{R}%
^{4}\times\mathbb{T}$ as%
\begin{align*}
\tilde{\rho}\left(  x\right)   &  =\inf\left\{  t>0:\tilde{\Phi}_{t}\left(
x\right)  \in\tilde{\Sigma}\right\}  ,\\
\mathcal{\tilde{P}}\left(  x\right)   &  =\tilde{\Phi}_{\tilde{\rho}\left(
x\right)  }\left(  x\right)  .
\end{align*}
Let also $\tilde{A}_{0}$ be a $5\times5$ matrix defined as
\begin{equation}
\tilde{A}_{0}=\left(
\begin{array}
[c]{cc}%
A_{0} & 0\\
0 & 1
\end{array}
\right)  , \label{eq:A0-tilde}%
\end{equation}
where $A_{0}$ is from (\ref{eq:A0-choice}). We define $\tilde{F}:\tilde
{\Sigma}\rightarrow\tilde{\Sigma}$ as
\[
\tilde{F}\left(  x,\theta\right)  =\tilde{A}_{0}^{-1}\left(  \mathcal{\tilde
{P}}^{2}\left(  \left(  x_{0},0\right)  +\tilde{A}_{0}\left(  x,\theta\right)
\right)  -\left(  x_{0},0\right)  \right)  .
\]
Then $\left\{  0\right\}  \times\mathbb{T}$ becomes a an invariant curve for
the map $\tilde{F}$. We will now show how to obtain bounds for the unstable
fibers of a point $\left(  0,\lambda\right)  $ on such curve, in the extended
phase space.

\begin{lemma}
\label{lem:fiber-extended}Let $r$ and $L$ be the constants considered in Lemma
\ref{lem:Wu-bound-CAP}. Let $M\in\mathbb{R}$, $M>0$. Consider cones in the
extended phase space defined as
\[
\tilde{Q}_{0}^{+}\left(  v\right)  =\left\{  v+\left(  t,tx_{2},tx_{3}%
,tx_{4},t\theta\right)  :x_{k}\in\left[  -L,L\right]  \text{ for
}k=2,3,4,\text{ }\theta\in\left[  -M,M\right]  \text{ and }t\in\mathbb{R}%
\right\}  .
\]
If for every $\lambda\in\mathbb{T}$ the unstable eigenvector of $D\tilde
{F}\left(  0,\lambda\right)  $ is contained in $\tilde{Q}_{0}^{+}\left(
0\right)  $ and if $\tilde{F}$ satisfies $\tilde{Q}_{0}^{+}$ cone conditions
on $\left[  -r,r\right]  ^{4}\times\mathbb{T}$, then for every $\lambda
\in\mathbb{T}$ the unstable fiber $W_{\left(  0,\lambda\right)  }^{u}(\tilde{F})$ is
parameterised by $\tilde{p}_{\lambda}^{\mathrm{u}}:\left[  -r,r\right]
\rightarrow\mathbb{R}^{4}\times\mathbb{T}$ satisfying (below $p^{\mathrm{u}}$
is the function established in Lemma \ref{lem:Wu-bound-CAP})
\begin{equation}
\pi_{\mathbb{R}^{4}}\tilde{p}_{\lambda}^{\mathrm{u}}\left(  u\right)
=p^{\mathrm{u}}\left(  u\right)  \qquad\text{and \qquad}\tilde{p}_{\lambda
}^{\mathrm{u}}\left(  u\right)  \in\tilde{Q}_{0}^{+}\left(  0,\lambda
\right)  \qquad \mbox{for } u\in[-r,r]. \label{eq:fibers-extended-1}%
\end{equation}
In particular%
\begin{equation}
\left\vert \pi_{\theta}\tilde{p}_{\lambda}^{\mathrm{u}}\left(  u\right)
-\lambda\right\vert \leq rM. \label{eq:fibers-extended-2}%
\end{equation}
For the family of Lyapunov orbits with $\mathcal{X}$ from the interval (\ref{eq:X-interval}) we can take $M=3.$
\end{lemma}

\begin{proof}
For every $\lambda$ the unstable fiber $W_{\left(  0,\lambda\right)  }^{u}(\tilde{F})$ is
tangent at $\left(  0,\lambda\right)  $ to the eigenvector of $D\tilde
{F}\left(  0,\lambda\right)  .$ This implies that sufficiently close to
$\left(  0,\lambda\right)  $ this fiber is
in $\tilde{Q}_{0}^{+}\left(  0,\lambda\right)  $. Every point $x\in W_{\left(
0,\lambda\right)  }^{u}(\tilde{F})$ can be expressed as $x=\tilde{F}^{m}\left(  z\right)
$ for an arbitrary $m\in\mathbb{N}$ , and for some appropriate point $z=z(m)\in
W_{\tilde{F}^{-m}(0,\lambda)}^{u}(\tilde{F})$. By taking sufficiently large $m$ the point $z$ can
be chosen as close to $\left\{  0\right\}  \times\mathbb{T}$ as we want, which
means that we can choose $m$ large enough so that $z\in\tilde{Q}_{0}%
^{+}\left(  \tilde{F}^{-m}(0,\lambda)\right)  .$ Since $z\in\tilde{Q}_{0}^{+}\left(
\tilde{F}^{-m}(0,\lambda)\right)  $, by the fact that $\tilde{F}$
satisfies cone condition we obtain that $x=\tilde{F}%
^{m}\left(  z\right)  \in\tilde{Q}_{0}^{+}\left(  0,\lambda\right)  .$ This
means that for $W_{\left(  0,\lambda\right)  }^{u}(\tilde{F})\cap (  \left[
-r,r\right]  ^{4}\times\mathbb{T} )  \subset\tilde{Q}_{0}^{+}\left(
0,\lambda\right)  $. Since the PCR3BP is autonomous, $\pi_{\mathbb{R}^{4}%
}W_{\left(  0,\lambda\right)  }^{u}(\tilde{F})$ does not depend on $\lambda$, and is
parameterised by $p^{\mathrm{u}}\left(  u\right)  $ from Lemma
\ref{lem:Wu-bound-CAP}. We have thus established (\ref{eq:fibers-extended-1}).

The condition (\ref{eq:fibers-extended-2}) follows from the fact that
$\tilde{p}_{\lambda}^{\mathrm{u}}\left(  u\right)  \in\tilde{Q}_{0}^{+}\left(
0,\lambda\right)  $ by computing
\[
\left\vert \pi_{\mathbb{T}}\tilde{p}_{\lambda}^{\mathrm{u}}\left(  u\right)
-\lambda\right\vert \leq M\left\vert \pi_{x_{1}}\tilde{p}_{\lambda
}^{\mathrm{u}}\left(  u\right)  \right\vert =M\left\vert \pi_{x_{1}%
}p^{\mathrm{u}}\left(  u\right)  \right\vert =M\left\vert u\right\vert \leq
rM.
\]

With computer assistance, we have validated that for every $\mathcal{X}$ from the
interval (\ref{eq:X-interval}), the unstable eigenvector of
$D\tilde{F}\left(  0,\lambda\right)  $ is contained in
\[
\{1\}\times\lbrack-3\cdot10^{-6},3\cdot10^{-6}]^{3}\times\lbrack
2.7763408157, 2.7918430312]\subset\tilde{Q}_{0}^{+}(0).
\]
We have also validated the cone conditions using the method described in
Remark \ref{rem:cone-propagation}.
\end{proof}

\begin{remark}
By choosing $\tilde{A}_{0}$ more carefully, instead of just adding $1$ in the
lower diagonal term in (\ref{eq:A0-tilde}), we could reduce the constant $M$
in Lemma \ref{lem:fiber-extended}. We have opted for (\ref{eq:A0-tilde}) for
the sake of simplicity, since the established $M$ is good enough for the proof
of the main result.
\end{remark}


\subsection{Intersection of stable/unstable manifolds of Lyapunov
orbits\label{sec:inter}}

The way we establish intersection of the stable and unstable manifolds of
Lyapunov orbits is similar to the method from section \ref{sec:Lyap}. We use
the parallel shooting from section \ref{sec:shooting} combined with the bounds
on the unstable manifold established in section \ref{sec:local-unstable}. We
fix $\mathcal{X}$ from (\ref{eq:X-interval}) and consider $p:\mathbb{R}%
\rightarrow\mathbb{R}^{4}$ for the shooting operator (\ref{eq:F-for-shooting})
to be
\begin{equation}
p\left(  s\right)  =x_{0}(\mathcal{X})+A_{0}p^{\mathrm{u}}\left(  s\right)
,\label{eq:p-s-curve}%
\end{equation}
where $p^{\mathrm{u}}$ is the function from which parameterises the
intersection of the unstable manifold of $L_{\mathcal{X}}$ with $\left\{
Y=0\right\}  $, whose bounds we have obtained in Lemma \ref{lem:Wu-bound-CAP}.
(The choice of $A_{0}$ and the bounds on $p^{\mathrm{u}}\left(  s\right)  ,$
together with its derivative are written out in Lemma \ref{lem:Wu-bound-CAP}).

\begin{table}
\begin{center}
{\scriptsize 
\begin{tabular}{l*{4}{l}}
\normalsize{$n$} & \normalsize{$X$} & \normalsize{$Y$} & \normalsize{$P_X$} & \normalsize{$P_Y$} \\
\hline \\
0 & -0.95000000242 & -0 & -1.0427994645e-08 & -0.84134724275 \\ 
1 & -0.94997599415 & 0.032508255048 & -0.026407880234 & -0.87841641186 \\ 
2 & -0.94367569216 & 0.046559121408 & -0.016859552423 & -0.93400981538 \\ 
3 & -0.93185570859 & 0.039269494574 & -0.0043284590814 & -0.98260192685 \\ 
4 & -0.9227740512 & 0.014134265524 & 0.00018569728993 & -1.0132585163 \\ 
5 & -0.92342538156 & -0.017741502784 & -8.9504023812e-05 & -1.0111198858 \\ 
6 & -0.9332873018 & -0.041191398464 & 0.0054636769818 & -0.97749478375 \\ 
7 & -0.94483412083 & -0.046015557539 & 0.018549243384 & -0.92767343656 \\ 
8 & -0.95017205833 & -0.029473079334 & 0.025907589742 & -0.87187283933 \\ 
9 & -0.95002477482 & 0.0043928225728 & -0.0053602054156 & -0.84205223945 \\ 
10 & -0.94968890529 & 0.035271238551 & -0.026402429149 & -0.88508159957 \\ 
11 & -0.94246936536 & 0.046807227358 & -0.015244331177 & -0.94027843983 \\ 
12 & -0.93053821073 & 0.037102593271 & -0.003488949681 & -0.9875100719 \\ 
13 & -0.92242610924 & 0.010464221796 & -0.00019611370513 & -1.0149474376 \\ 
14 & -0.92456420774 & -0.021092391132 & -0.00083603746 & -1.0083700678 \\ 
15 & -0.93548053074 & -0.042380742641 & 0.0047663014948 & -0.97130987256 \\ 
16 & -0.94730307615 & -0.043772642856 & 0.016236744567 & -0.91841934035 \\ 
17 & -0.95320906454 & -0.022318932226 & 0.012564889629 & -0.85729285711 \\ 
18 & -0.96070049024 & 0.017216683341 & -0.061196504134 & -0.84499124342 \\ 
19 & -0.98081741509 & 0.044653406262 & -0.11087155568 & -0.94184521723 \\ 
20 & -1.0057774379 & 0.042684693077 & -0.12465612254 & -1.0610257866 \\ 
21 & -1.0281206713 & 0 & 0 & -1.2112777154 \\ 
\end{tabular}
}\end{center}
\caption{Midpoints of our enclosure o homoclinic orbits for our family of Lyapunov orbits.\label{tabl:homoclinic}}
\end{table}

By the $\mathcal{R}$-symmetry of the PCR3BP we know that $\mathcal{R}\left(
p\left(  s\right)  \right)  $ is a point on the stable manifold of
$L_{\mathcal{X}}$. If for $F$ defined by (\ref{eq:F-for-shooting}) we validate
that for some $s\in\left(  0,1\right)  $ and $h\in\mathbb{R}$ we have
\[
F\left(  s,h,x_{1},\ldots,x_{n}\right)  =0,
\]
then taking $x_{0}:=p\left(  s\right)  $ and $x_{n+1}:=\Phi_{h}\left(
x_{n}\right)  $ we obtain a sequence of points%
\begin{equation}
x_{0},\ldots,x_{n},x_{n+1},\mathcal{R}\left(  x_{n}\right)  ,\ldots
,\mathcal{R}\left(  x_{0}\right),  \label{eq:points-along-homoclinic}%
\end{equation}
along an intersection of the stable and unstable manifolds of $L_{\mathcal{X}%
}$. With this method we have managed to obtain the following result.

\begin{lemma}
\label{lem:homoclinic-bound}For every $\mathcal{X}$ from (\ref{eq:X-interval})
the stable and unstable manifolds of $L_{\mathcal{X}}$ intersect. Moreover,
the intersection is along a $\mathcal{R}$-symmetric homoclinic orbit, which
contains a sequence of points along it that lies $r=1.96\cdot10^{-7}$ close
to the orbit written out in Table \ref{tabl:homoclinic}; see also Figure \ref{fig:homoclinic}.
\end{lemma}

\begin{remark}
In Table \ref{tabl:homoclinic} we write out half of the points along the
homoclinic, since the second half (the remaining twenty one points, to be
precise) follows from the $\mathcal{R}$-symmetry.
\end{remark}

\begin{remark}
Our estimate on the distance from the points from Table \ref{tabl:homoclinic}
obtained by our computer program varies from point to point and is frequently
more accurate than stated in Lemma \ref{lem:homoclinic-bound}, where we simply
write out the single $r$, which is the upper bound that can be applied to all
the points.
\end{remark}

\begin{remark}
\label{rem:time-between-points} From the method we also obtain a bound on the
integration time $h$ between the consecutive points along the homoclinic. We
have obtained that
\begin{equation}
h\in[0.34246881126, 0.34246888642].\label{eq:h-bound}%
\end{equation}

\end{remark}

\begin{figure}
\begin{center}
\includegraphics[height=4cm]{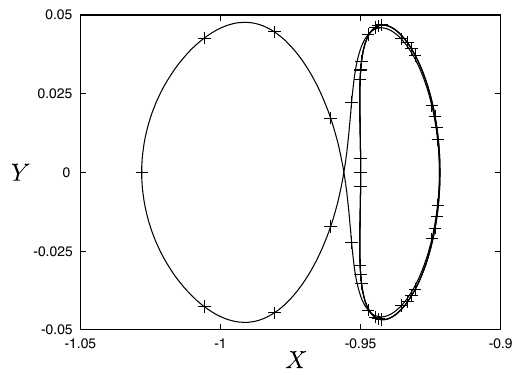}
\end{center}
\caption{A homoclinic orbit to $L_{\mathcal{X}}.$\label{fig:homoclinic}}
\end{figure}

We now show that the established intersection is transversal.

\begin{lemma}
\label{tem:inter-transversal} For every $\mathcal{X}$ from the interval
(\ref{eq:X-interval}) the intersection of the stable and unstable manifolds of
$L_{\mathcal{X}}$ is transversal\correction{comment 28}{} when considered in the three dimensional
constant energy level.
\end{lemma}

\begin{proof}
We consider $\Sigma_{\{X<-1\}}=\left\{  Y=0,X<-1\right\}  \subset
\mathbb{R}^{4}$ and we will study the intersections of the stable/unstable
manifolds of $L_{\mathcal{X}}$ on this section. (See Figure \ref{fig:homoclinic}.)

Let us denote by $W_{L_{\mathcal{X}}}^{u}$ and $W_{L_{\mathcal{X}}}^{s}$ the
unstable and stable manifolds of $L_{\mathcal{X}},$ respectively. These are
two dimensional tubes, contained in the three dimensional constant energy
level $\left\{  H=h^{\ast}\right\}  $ for $h^{\ast}=H\left(  L_{\mathcal{X}%
}\right)  $. We have established that $W_{L_{\mathcal{X}}}^{u}$ and
$W_{L_{\mathcal{X}}}^{s}$ intersect along a homoclinic orbit, which passes
through the points (\ref{eq:points-along-homoclinic}). The point $x_{n+1}$
belongs to $\Sigma_{\{X<-1\}}$. (From (\ref{eq:F-for-shooting}) we know that
$\pi_{Y}x_{n+1}=0$ and from Table \ref{tabl:homoclinic} we see that $\pi
_{X}x_{n+1}<-1$, so $x_{n+1}\in\Sigma_{\{X<-1\}}$.) The vector field at
$x_{n+1}$ has a non zero $Y$-component. This means that the tangent spaces to
the stable and unstable manifolds at $x_{n+1}$ span the coordinate $Y$.

The manifolds $W_{L_{\mathcal{X}}}^{u}$ and $W_{L_{\mathcal{X}}}^{s}$
intersect with $\Sigma_{\{X<-1\}}$ at $x_{n+1}$ along one dimensional curves.
(Note that some of the points from the unstable/stable manifolds can collide
with Jupiter. Those that reach $\Sigma_{\{X<-1\}}$ close to $x_{n+1}$
intersect the section along one dimensional curves.) The section
$\Sigma_{\{X<-1\}}$ is three dimensional, but $L_{\mathcal{X}},$
$W_{L_{\mathcal{X}}}^{u}$ and $W_{L_{\mathcal{X}}}^{s}$ are contained in
$\left\{  H=h^{\ast}\right\}  $. The set $\Sigma_{\{X<-1\}}\cap\left\{
H=h^{\ast}\right\}  $ is two dimensional and can be parameterised\footnote{On
$\Sigma_{\left\{  X<-1\right\}  }$ the coordinate $P_{Y}$ can be computed from $X,P_X$ since 
$H(X,Y=0,P_X,P_Y)=h^*  $.} by coordinates $(X,P_{X})$. The
$W_{L_{\mathcal{X}}}^{u}\cap\Sigma_{\{X<-1\}}$ and $W_{L_{\mathcal{X}}}%
^{s}\cap\Sigma_{\{X<-1\}}$ are therefore one dimensional curves contained in a
two dimensional space, parameterised by $(X,P_{X})$, and if we show that 
\begin{equation}
\pi_{X,P_{X}}\left(  W_{L_{\mathcal{X}}}^{u}\cap\Sigma_{\{X<-1\}}\right)  \text{\quad
intersect transversally with\quad}\pi_{X,P_{X}}\left(  W_{L_{\mathcal{X}}}%
^{s}\cap\Sigma_{\{X<-1\}}\right)  \text{\quad at }x_{n+1}%
,\label{eq:transversal-curves}%
\end{equation}
then we will obtain transversal intersections of $W_{L_{\mathcal{X}}}^{u}$
with $W_{L_{\mathcal{X}}}^{s}$ at $x_{n+1}$ in $\left\{  H=h^{\ast}\right\}
.$ (In more detail: the vector field at $x_{n+1}$ is tangent to
$W_{L_{\mathcal{X}}}^{u}$ and $W_{L_{\mathcal{X}}}^{s}$ at $x_{n+1}$ and has a
non zero $Y$-component; from (\ref{eq:transversal-curves}) we will have that
the tangent spaces to $W_{L_{\mathcal{X}}}^{u}$ and $W_{L_{\mathcal{X}}}^{s}$
at $x_{n+1}$ span $X,P_{X}$. In all we span a three dimensional vector space,
hence the intersection is transversal in $\left\{  H=h^{\ast}\right\}  .$)

Let $\tau:\mathbb{R}^{4}\rightarrow\mathbb{R}$, $\mathcal{P}:\mathbb{R}%
^{4}\rightarrow\Sigma_{\{X<-1\}}$ be defined as
\begin{align*}
\tau\left(  x\right)   &  :=\inf\left\{  t>0:\Phi_{t}\left(  x\right)
\in\Sigma_{\left\{  X<-1\right\}  }\right\}  ,\\
\mathcal{P}\left(  x\right)   &  :=\Phi_{\tau\left(  x\right)  }\left(
x\right)  .
\end{align*}
The curve $W_{L_{\mathcal{X}}}^{u}\cap\Sigma_{\{X<-1\}}$ can be obtained by
computing $\mathcal{P}\circ p\left(  s\right)  $. (See (\ref{eq:p-s-curve})
for the definition of $p\left(  s\right)  .$) By the $\mathcal{R}$-symmetry of
the PCR3BP the $W_{L_{\mathcal{X}}}^{s}\cap\Sigma_{\{X<-1\}}$ is equal to
$\mathcal{R}\circ\mathcal{P}\circ p\left(  s\right)  $. Let $s^{\ast}%
\in\mathbb{R}$ be such that $\mathcal{P}\circ p\left(  s^{\ast}\right)
=x_{n+1}$. If we establish that
\begin{equation}
\pi_{X}\frac{d}{ds}\mathcal{P}\circ p\left(  s\right)  |_{s=s^{\ast}}%
>0\qquad\text{and\qquad}\pi_{P_{X}}\frac{d}{ds}\mathcal{P}\circ p\left(
s\right)  |_{s=s^{\ast}}>0,\label{eq:trans-cond-1}%
\end{equation}
then%
\begin{align}
\pi_{X}\frac{d}{ds}\mathcal{R}\circ\mathcal{P}\circ p\left(  s\right)
|_{s=s^{\ast}} &  =\pi_{X}\frac{d}{ds}\mathcal{P}\circ p\left(  s\right)
|_{s=s^{\ast}}>0,\label{eq:trans-cond-2}\\
\pi_{P_{X}}\frac{d}{ds}\mathcal{R}\circ\mathcal{P}\circ p\left(  s\right)
|_{s=s^{\ast}} &  =-\pi_{P_{X}}\frac{d}{ds}\mathcal{P}\circ p\left(  s\right)
|_{s=s^{\ast}}<0,\label{eq:trans-cond-3}%
\end{align}
and from (\ref{eq:trans-cond-1}--\ref{eq:trans-cond-3}) we will obtain
(\ref{eq:transversal-curves}).

In section \ref{sec:local-unstable} we have established that (see Lemma
\ref{lem:Wu-bound-CAP} and Corollary \ref{cor:ps-der})
\begin{equation}
\frac{d}{ds}p\left(  s\right)  \in\tilde{A}_{0}V_{0}, \label{eq:ps-cone-bound}%
\end{equation}
where $V_{0}$ is a set%
\[
V_{0}=\left\{  1\right\}  \times\left[  -L,L\right]  \times\left[
-L,L\right]  \times\left[  -L,L\right]  ,
\]
with $L=5\cdot10^{-6}$. The set $V_{0}$ represents a cone, as described in
Remark \ref{rem:cone-rep}. We can propagate the bound (\ref{eq:ps-cone-bound})
to the point $x_{n+1}$ using cone propagation method described in Remark
\ref{rem:cone-propagation}. We have thus validated that%
\[
\pi_{X,P_{X}}\frac{d}{ds}\mathcal{P}\circ p\left(  s\right)  |_{s=s^{\ast}}%
\in\left\{  uV\,|\,u>0\text{ and }V=\{1\}\times\lbrack
4.06081, 4.06404]\right\}  .
\]
This establish (\ref{eq:trans-cond-1}) and finishes our proof.
\end{proof}


\subsection{Persistence of the family of Lyapunov
orbits\label{sec:persistence}}

Recall that a Lyapunov orbit starting from $x_{0}(\mathcal{X})$ (see
(\ref{eq:x0-formula})), which has a period $T(\mathcal{X})$, is given as
\[
L_{\mathcal{X}}=\left\{  \Phi_{t}\left(  x_{0}(\mathcal{X})\right)
:t\in\left[  0,T(\mathcal{X})\right]  \right\}  .
\]
Let us denote the normally hyperbolic invariant manifold consisting of the
family of Lyapunov orbits as
\[
\Lambda_{L}=\{L_{\mathcal{X}}:\mathcal{X}%
\mbox{ are from the interval (\ref{eq:X-interval})}\}\subset\mathbb{R}^{4}.
\]
We shall also use the following notation for the manifold in the extended
phase space:
\begin{equation}
\tilde{\Lambda}_{L}=\Lambda_{L}\times\mathbb{T}.\label{eq:nhim-lyap}%
\end{equation}

To prove persistence of $\tilde{\Lambda}_{L}$ we shall use the following theorem:

\begin{theorem}
\cite{MR3604613}\label{th:persistence} Assume that
\begin{equation}
\frac{d}{d\mathcal{X}}T(\mathcal{X})\neq0\label{eq:dTdX}%
\end{equation}
and also
\begin{equation}
\frac{d}{d\mathcal{X}}H_0(x_{0}(\mathcal{X}))\neq0.\label{eq:dHdX}%
\end{equation}
Then for sufficiently small perturbation $\varepsilon$ from the PCR3BP to the
PER3BP, the manifold $\tilde{\Lambda}_{L}$ is perturbed into a $O\left(
\varepsilon\right)  $ close normally hyperbolic manifold $\tilde{\Lambda}%
_{L}^{\varepsilon}$, with boundary, which is invariant under the flow induced
by (\ref{eq:PER3BP-ode}). Moreover, there exists a Cantor set of invariant
tori in $\tilde{\Lambda}_{L}^{\varepsilon}$.\correction{comment 29}{}
\end{theorem}

We now discuss how we validate (\ref{eq:dTdX}--\ref{eq:dHdX}). 

Let us fix a
single $\mathcal{X}$ from (\ref{eq:x0-formula}). \cor{Let us first recall that }from section \ref{sec:Lyap}
we know that there exists a $\tau=\tau(\mathcal{X})$, and an even\footnote{In our case we take $N=30$, see Table \ref{tabl:Lyap}.}
$N\in\mathbb{N}$, such that for $\tau=\tau(\mathcal{X})$ and for%
\begin{align*}
x_{0} &  =x_{0}\left(  \mathcal{X}\right)  =\left(  \mathcal{X},0,0,s\left(
\mathcal{X}\right)  \right)  ,\\
x_{i+1} &  =\Phi_{\tau}\left(  x_{i}\right)  \qquad\text{for }i=0,\ldots,N-1
\end{align*}
we have%
\[
x_{N}=x_{0},
\]
and the period of the Lyapunov orbit starting from $x_{0}$ is $T(\mathcal{X}%
)=N\tau\left(  \mathcal{X}\right)  $.
\correction{comment 9}{The existence of such $s=s(\mathcal{X})$ and $ \tau=\tau(\mathcal{X})$
}
was established by fixing $\mathcal{X}$ and solving for $\tau$ and $s$
the following equation
\begin{equation}
\pi_{Y,P_{X}}\Phi_{\tau N/2}\left(  \mathcal{X},0,0,s\right)  =0. \label{eq:for-twist}
\end{equation}
(See (\ref{eq:F-for-shooting}) and (\ref{eq:ps-for-Lyap}); equation (\ref{eq:F-for-shooting}) provides us with the solution of (\ref{eq:for-twist}) using parallel shooting.)

\cor{With the aim of validating (\ref{eq:dTdX}--\ref{eq:dHdX}) }we can now define a function $g:\mathbb{R}^{3}\rightarrow\mathbb{R}^{2}$ as%
\[
g\left(  \mathcal{X},s,\tau\right)  =\pi_{Y,P_{X}}\Phi_{N\tau/2}\left(
\mathcal{X},0,0,s\right)
\]
and observe that
\begin{equation}
g\left(  \mathcal{X},s(\mathcal{X}),\tau(\mathcal{X})\right)
=0.\label{eq:g-zero}%
\end{equation}
This means that we can compute the derivatives of $\frac{ds}{d\mathcal{X}}$
and $\frac{d\tau}{d\mathcal{X}}$ from the implicit function theorem; i.e.
by differentiating (\ref{eq:g-zero}) with respect to $\mathcal{X}$ we obtain
\[
\frac{\partial g}{\partial\mathcal{X}}+\frac{\partial g}{\partial\left(
s,\tau\right)  }\left(
\begin{array}
[c]{c}%
\frac{ds}{d\mathcal{X}}\\
\frac{d\tau}{d\mathcal{X}}%
\end{array}
\right)  =0,
\]
and provided that $\frac{\partial g}{\partial\left(  s,\tau\right)  }$ is
invertible we see that%
\begin{equation}
\left(
\begin{array}
[c]{c}%
\frac{ds}{d\mathcal{X}}\\
\frac{d\tau}{d\mathcal{X}}%
\end{array}
\right)  =-\left(  \frac{\partial g}{\partial\left(  s,\tau\right)  }\right)
^{-1}\frac{\partial g}{\partial\mathcal{X}}.\label{eq:dg-0}%
\end{equation}

The partials $\frac{\partial g}{\partial\mathcal{X}},$ $\frac{\partial
g}{\partial s}$ and $\frac{\partial g}{\partial\tau}$ are $2\times1$ matrices,
which can be computed as follows. Let $e_{i}\in\mathbb{R}^{4}$, for
$i=1,\ldots,4$, be a vector with $1$ on $i$-th coordinate and zeros on the
remaining coordinates. Let $\mathcal{F}:\mathbb{R}^4\to \mathbb{R}^4$ stand for the vector field of the PCR3BP. Then%
\begin{align}
\frac{\partial g}{\partial\mathcal{X}}\left(  x_{0}\right)   &  =\pi_{Y,P_{X}%
}D\Phi_{N\tau/2}\left(  x_{0}\right)  e_{1}=\pi_{Y,P_{X}}D\Phi_{\tau}\left(
x_{N/2-1}\right)  \ldots D\Phi_{\tau}\left(  x_{0}\right)  e_{1}%
,\label{eq:dg-1}\\
\frac{\partial g}{\partial s}\left(  x_{0}\right)   &  =\pi_{Y,P_{X}}%
D\Phi_{N\tau/2}\left(  x_{0}\right)  e_{4}=\pi_{Y,P_{X}}D\Phi_{\tau}\left(
x_{N/2-1}\right)  \ldots D\Phi_{\tau}\left(  x_{0}\right)  e_{4},\\
\frac{\partial g}{\partial\tau}\left(  x_{0}\right)   &  
=\frac{d}{d\tau} \pi_{X,P_X}\Phi_{N\tau/2}(x_0)
=\frac{N}{2}%
\pi_{Y,P_{X}}\mathcal{F}\left(  \Phi_{N\tau/2}\left(  x_{0}\right)  \right)  =\frac
{N}{2}\pi_{Y,P_{X}}\mathcal{F}\left(  x_{N/2}\right)  .\label{eq:dg-3}%
\end{align}
Note that from the above $\frac{\partial g}{\partial s}$ and $\frac{\partial
g}{\partial\tau}$ we obtain $\frac{\partial g}{\partial\left(  s,\tau\right)
}=\left(
\begin{array}
[c]{cc}%
\frac{\partial g}{\partial s} & \frac{\partial g}{\partial\tau}%
\end{array}
\right)$, which we can use in (\ref{eq:dg-0}) to compute $\frac{ds}{d\mathcal{X}}$ and $\frac{d\tau}{d\mathcal{X}}$.

Once $\frac{ds}{d\mathcal{X}}$ is established, we can easily compute
\begin{equation}
\frac{dH_0}{d\mathcal{X}}(x_{0}(\mathcal{X}))=\frac{\partial H_0}{\partial
X}(x_{0}(\mathcal{X}))+\frac{\partial H_0}{\partial P_{Y}}%
(x_{0}(\mathcal{X}))\frac{ds}{d\mathcal{X}}(\mathcal{X}%
).\label{eq:dH-along-x0}%
\end{equation}
We have used (\ref{eq:dg-0}--\ref{eq:dH-along-x0}) to validate, with computer
assistance, that we have the following:

\begin{lemma}
\label{lem:persistence}For every $\mathcal{X}$ from (\ref{eq:X-interval}) we
have%
\begin{align*} 
\frac{d}{d\mathcal{X}}x_{0}\left(  \mathcal{X}\right)   &  \in\{1\}\times
\{0\}\times\{0\}\times\left[  -4.530367,-4.530349\right]  ,\\
\frac{d}{d\mathcal{X}}\tau(\mathcal{X}) &  \in\left[
-0.5523403,-0.5522624\right]  ,\\
\frac{d}{d\mathcal{X}}T(\mathcal{X}) &  \in\left[  -16.57021,-16.56787\right]
,\\
\frac{d}{d\mathcal{X}}H_0\left(  x_{0}\left(  \mathcal{X}\right)  \right)   &
\in\left[  -0.359187,-0.359185\right]  .
\end{align*}

\end{lemma}

\begin{corollary}
\label{cor:persistence} By Theorem \ref{th:persistence} and Lemma
\ref{lem:persistence} we obtain the persistence of the normally hyperbolic
manifold $\tilde{\Lambda}_{L}$ (see (\ref{eq:nhim-lyap})), which consists of
our family of Lyapunov orbits in the extended phase space.
\end{corollary}


\subsection{Proof of the main theorem\label{sec:main-proof}}

Recall that $\Phi_{t}^{\varepsilon}$ is the flow of the PER3BP in the extended
phase space. Let $\rho^{\varepsilon}:\mathbb{R}^{4}\times\mathbb{T}%
\rightarrow\mathbb{R}$ be the time to the section $\left\{  Y=0\right\}  $
\[
\rho^{\varepsilon}\left(  x\right)  =\inf\left\{  t>0:\Phi_{t}^{\varepsilon
}\left(  x\right)  \in\left\{  Y=0\right\}  \right\}  ,
\]
and let $\mathcal{P}_{\varepsilon}:\mathbb{R}^{4}\times\mathbb{T}\rightarrow\left\{  Y=0\right\}  $ be defined as
\begin{equation}
\mathcal{P}_{\varepsilon}\left(  x\right)  =\Phi_{\rho\left(  x\right)
}^{\varepsilon}\left(  x\right)  . \label{eq:P-eps-map}%
\end{equation}
The section $\left\{  Y=0\right\}  $ in the extended phase space is isomorphic
with $\mathbb{R}^{3}\times\mathbb{T}$, which fits the setting from section
\ref{sec:diff-mech}.

We consider a single point $x^{\ast}:=x_{0}\left(  -0.95\right)  $
($x_{0}(\mathcal{X})$ is defined in (\ref{eq:x0-formula}); see also
(\ref{eq:X-interval}) regarding the choice of $-0.95$) and consider the matrix
$\tilde{A}_{0}$ from (\ref{eq:A0-tilde}). We define%
\[
\tilde{\Sigma}=\left\{  \tilde{A}_{0}^{-1}\left(  x-x^{\ast},\theta\right)
:x\in\left\{  Y=0\right\}  ,\theta\in\mathbb{T}\right\}  .
\]
In other words, $\tilde{\Sigma}$ is the section $\left\{  Y=0\right\}  $
considered in the extended phase space, in the local coordinates given by the
affine change given by $x^{\ast}$ and $\tilde{A}_{0}$. We decide to work in
these local coordinates since then we can directly use the estimates on the local
unstable manifolds, which were established in section \ref{sec:local-unstable}%
. This is the reason why we choose $\tilde{\Sigma}$ as above.

To apply Theorem \ref{th:shadowing-seq} we will choose our family of maps
(\ref{eq:maps-for-diff-mech})%
\[
f_{\varepsilon}:\tilde{\Sigma}\rightarrow\tilde{\Sigma}%
\]
to be defined as
\begin{equation}
f_{\varepsilon}\left(  x,\theta\right)  =\tilde{A}_{0}^{-1}\left(
\mathcal{P}_{\varepsilon}\circ\mathcal{P}_{\varepsilon}\left(  \left(
x^{\ast},0\right)  +\tilde{A}_{0}\left(  x,\theta\right)  \right)  -\left(
x^{\ast},0\right)  \right)  . \label{eq:return-map}%
\end{equation}
In other words, we consider the return map to the section $\{Y=0\}$ expressed
in our local coordinates.

\begin{remark}\label{two-poincare-composition}\correction{comment 10}{
A single Lyapunov orbit becomes a two dimensional torus in the extended phase space. (See Figure \ref{fig:Lambda_0}; Left.) The intersection of such torus with the section $\{Y=0\}$ constitutes two disjoints circles.  Each of these is an invariant circle under $\mathcal{P}_{\varepsilon}\circ\mathcal{P}_{\varepsilon}$. For instance $\{x^*\} \times \mathbb{T}^1$ is one of such invariant circles. The circle $\{x^*\} \times \mathbb{T}^1$, which is invariant under under $\mathcal{P}_{\varepsilon}\circ\mathcal{P}_{\varepsilon}$, corresponds to the circle $\{0\}\times \mathbb{T}^1 \subset \tilde \Sigma $, which is invariant under the map $f_{\varepsilon=0}$.} 



\end{remark}

Prior to the perturbation, for $\varepsilon=0$, we define our normally hyperbolic invariant manifold $\Lambda_{0}\subset
\tilde{\Sigma}$ as (see (\ref{eq:nhim-lyap}) for the definition of
$\tilde{\Lambda}_{L}$)
\[
\Lambda_{0}:=\tilde{\Lambda}_{L}\cap\tilde{\Sigma}.
\]
(See Figure \ref{fig:Lambda_0}; Middle.) The manifold $\Lambda_{0}\ $is foliated by invariant circles for the map
$f_{0}$. \begin{figure}[ptb]
\begin{center}
\includegraphics[height=3.5cm]{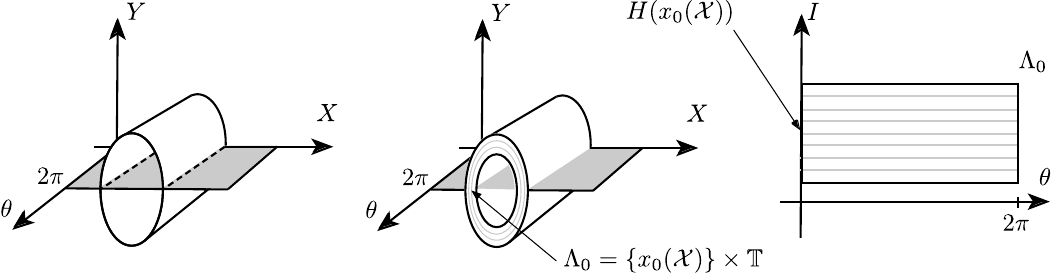}
\end{center}
\caption{\cor{Left: A Lyapuov orbit in the extended phase space is a torus, which intersected with $\{Y=0\}$ gives two circles denoted by dashed lines. Middle: The manifold $\Lambda_{0}\subset \{Y=0\}$ is one of the two components of the intersection of a family of Laypunov orbits in the extended phase space with $\{Y=0\}$. Right: $\Lambda_{0}$ in coordinates $(I,\theta)$.}}%
\label{fig:Lambda_0}%
\end{figure}Note that for $\varepsilon=0$ the energy is preserved, so for
\begin{equation}
I(x):=H_{0}\left(  \left(  x^{\ast},0\right)  +\tilde{A}_{0} x \right)
\label{eq:energy-in-proof}%
\end{equation}
we see that%
\begin{equation}
I\left(  f_{0}\left(  x\right)  \right)  =I(x).
\label{eq:energy-preserved-in-proof}%
\end{equation}

We treat $I$ as our conserved variable when $\varepsilon=0$ and write
\[
\pi_{I}x=I(x).
\]
Thus, condition (\ref{eq:energy-preserved-in-proof}) plays the role of
(\ref{eq:energy-preserved}).

We can write $f_{\varepsilon}$ as%
\[
f_{\varepsilon}\left(  x\right)  =f_{0}\left(  x\right)  +\varepsilon g\left(
\varepsilon,x\right)
\]
where%
\begin{align}
\cor{g\left(  \varepsilon,x\right) }& \cor{= \frac{\partial f_{\varepsilon}}{\partial \varepsilon}(x)|_{\varepsilon=0}} \notag \\
g\left(  \varepsilon,x\right) & =\frac{1}{\varepsilon}\left(  f_{\varepsilon
}\left(  x\right)  -f_{0}\left(  x\right)  \right)   \qquad \cor{\mbox{for }\varepsilon \ne 0.} \label{eq:g-form}
\end{align}
To compute the change of the energy after an iterate of $f_{\varepsilon
}\left(  x\right)  $ we compute
\[
\pi_{I}f_{\varepsilon}\left(  x\right)  =\pi_{I}f_{0}\left(  x\right)
+\varepsilon\pi_{I}g\left(  \varepsilon,x\right)  =\pi_{I}x+\varepsilon\pi
_{I}g\left(  \varepsilon,x\right)  ,
\]
where $\pi_{I}g\left(  \varepsilon,x\right)  $ is
\[
\pi_{I}g\left(  \varepsilon,x\right)  =\frac{1}{\varepsilon}\left[  I\left(
f_{\varepsilon}\left(  x\right)  \right)  -I\left(  f_{0}\left(  x\right)
\right)  \right]  =\frac{1}{\varepsilon}\left[  I\left(  f_{\varepsilon
}\left(  x\right)  \right)  -I\left(  x\right)  \right]  .
\]
It follows from the above that%
\begin{equation}
\pi_{I}g\left(  0,x\right)  =\frac{d}{d\varepsilon}I(f_{\varepsilon
}(x))|_{\varepsilon=0}=DI\left(  f_{0}\left(  x\right)  \right)
\frac{\partial f_{0}}{\partial\varepsilon}\left(  x\right)  =\nabla I\left(
f_{0}\left(  x\right)  \right)  \cdot\frac{\partial f_{0}}{\partial
\varepsilon}\left(  x\right)  . \label{eq:gI-at-zero}%
\end{equation}
(In the above equation $\cdot$ is the scalar product.)

We are now ready for the proof of our main result.

\begin{proof}[Proof of Theorem \ref{th:main}]
We start by describing our system for $\varepsilon=0$. While discussing the
system for $\varepsilon=0$ we will recall some results for the PCR3BP established in the previous sections. We need
to keep in mind that these were considered in coordinates $\left(
X,Y,P_{X},P_{Y}\right)  \in\mathbb{R}^{4}$; without the extended time
coordinate $\theta$.

The manifold $\Lambda_{0}$ is invariant under $f_{0}$ and in coordinates
$X,Y,P_{X},P_{Y},\theta$ can be written as
\[
\Lambda_{0}=\left\{  \left(  x_{0}\left(  \mathcal{X}\right)  ,\theta\right)
:\text{ }\mathcal{X}\text{ is from (\ref{eq:X-interval}), } \theta
\in\mathbb{T}\right\}  \subset\mathbb{R}^{4}\times\mathbb{T}.
\]
For $\varepsilon=0$ the inner dynamics produced by $f_{0}$ on the manifold is
given as
\begin{equation}
\left(  x_{0}\left(  \mathcal{X}\right)  ,\theta\right)  \mapsto\left(
x_{0}\left(  \mathcal{X}\right)  ,\theta+T\left(  \mathcal{X}\right)  \right)
,\label{eq:inner-dynamics}%
\end{equation}
where $T\left(  \mathcal{X}\right)  $ is the period of the Lyapunov orbit
$L_{\mathcal{X}}$. Thus $\Lambda_{0}$ is an invariant cylinder, foliated by
invariant curves.

Recall that for a given single Lyapunov orbit $L_{\mathcal{X}}$, we have
established in Lemma \ref{lem:Wu-bound-CAP} the bounds on a curve
$p^{\mathrm{u}}(u)$, with $u\in\left(  -r,r\right)  $, which lies along the
intersection of the two dimensional local  unstable manifold of $L_{\mathcal{X}%
}$ (for the PCR3BP in $\mathbb{R}^{4}$) with the section $\{Y=0\}$. Let us
emphasize the dependence of $p^{\mathrm{u}}(u)$ on $\mathcal{X}$ by writing
$p_{\mathcal{X}}^{\mathrm{u}}\left(  u\right)  $. The unstable manifold of
$\Lambda_{0}$ for $f_{0}$, considered in $\{Y=0\}$ (in the extended phase
space) is three dimensional, and in the coordinates $X,P_{X},P_{Y},\theta$,
can locally be written as%
\[
W_{\Lambda_{0}}^{u}=\left\{  \left(  p_{\mathcal{X}}^{\mathrm{u}}\left(
u\right)  ,\theta\right)  :\text{ }\mathcal{X}\text{ is from
(\ref{eq:X-interval}), }u\in\left(  -r,r\right)  ,\theta\in\mathbb{T}\right\}
.
\]
By considering the $\mathcal{R}$-symmetry of the PCR3BP, in the extended phase
space, and restricted to $\{Y=0\}$, i.e.
\[
\mathcal{\tilde{R}}\left(  X,P_{X},P_{Y},\theta\right)  :=\left(
X,-P_{X},P_{Y},\theta\right)  ,
\]
we obtain the local stable manifold%
\[
W_{\Lambda_{0}}^{s}=\mathcal{\tilde{R}}\left(  W_{\Lambda_{0}}^{u}\right)  .
\]

We will now show that for $\varepsilon=0$ we have a well defined 
scattering map \correction{comment 12}{}
\begin{equation}
\sigma:\Lambda_{0}\rightarrow\Lambda_{0}.\label{eq:scatter-unperturbed}
\end{equation}
For this we first need to establish a homoclinic channel. By Lemmas
\ref{lem:homoclinic-bound}, \ref{tem:inter-transversal} we know that the two
dimensional stable and unstable manifolds of $L_{\mathcal{X}}$ in the PCR3BP
in $\mathbb{R}^{4}$ intersect transversally (when considered on a three
dimensional fixed energy set $\left\{  H_{0}=I\right\}  $, where
$I=H_{0}\left(  L_{\mathcal{X}}\right)  $) along an $\mathcal{R}$-symmetric
homoclinic orbit, which contains a point which we shall denote here as
$\gamma\left(  I\right)  \in\{Y=0\}$. The two dimensional stable and unstable
manifolds of a given Lyapunov orbit $L_{\mathcal{X}}$, when intersected with
$\{Y=0\}$, become one dimensional curves in $\{Y=0\}$ which intersect at
$\gamma\left(  I\right)  $. Let us denote these curves as $w_{I}^{\mathrm{u}%
}\left(  u\right)  $ and $w_{I}^{\mathrm{s}}\left(  s\right)  $, and work
under a convention that $w_{I}^{\mathrm{u}}\left(  0\right)  =w_{I}%
^{\mathrm{s}}\left(  0\right)  =\gamma\left(  I\right)  $. (This can always be
ensured by re-parameterising the curves.) We have added the subscript $I$ to
emphasize the dependence of the curves on the choice of the energy level: on
different energy levels we have a different Lyapunov orbits, that lead to
different curves. We have shown during the proof of Lemma
\ref{tem:inter-transversal} that the tangent vectors to these curves span the
$X,P_{X}$ plane, i.e.
\begin{equation}
\mathrm{span}\left(  \pi_{X,P_{X}}\frac{d}{du}w_{I}^{\mathrm{u}}(u)|_{u=0}%
,\pi_{X,P_{X}}\frac{d}{ds}w_{I}^{\mathrm{s}}(s)|_{s=0}\right)  =\mathbb{R}%
^{2}. \label{eq:X-PX-Trans}%
\end{equation}

We shall take
\begin{equation}
\Gamma:=\left\{  \left(  \gamma\left(  I\right)  ,\theta\right)  :\text{
}I=H_{0}\left(  L_{\mathcal{X}}\right)  \text{, }\mathcal{X}\text{ is from
(\ref{eq:X-interval}) and }\theta\in\mathbb{T}\right\}  \subset\tilde{\Sigma}
\label{eq:Gamma-def}%
\end{equation}
and prove that it is a well defined homoclinic channel by showing
(\ref{eq:scatter-cond-1}--\ref{eq:scatter-cond-4}).

It will be convenient for us to check the transversality conditions
(\ref{eq:scatter-cond-1}--\ref{eq:scatter-cond-4}) in coordinates
$X,P_{X},I,\theta$. In these coordinates we can parameterise $\Lambda_{0}$ by
$I,\theta$ (see Figure \ref{fig:Lambda_0})%
\[
\Lambda_{0}=\left\{  \left(  \pi_{X,P_{X}}x_{0}\left(  \mathcal{X}\right)
,I,\theta\right)  :I\in\mathbb{R},\text{ }\theta\in\mathbb{T},\,H_{0}\left(
L_{\mathcal{X}}\right)  =I\right\}  .
\]
Close the intersection of the stable and unstable manifold at $\left(
\gamma\left(  I\right)  ,\theta\right)  $, we can parameterise the manifold
$W_{\Lambda_{0}}^{u}$ by $I,\theta,u$ as follows%
\[
W_{\Lambda_{0}}^{u}=\left\{  \left(  \pi_{X,P_{X}}w_{I}^{\mathrm{u}}\left(
u\right)  ,I,\theta\right)  :I\in\mathbb{R},\text{ }\theta\in\mathbb{T}%
,\,u\in\mathbb{R}\right\}  .
\]
We can similarly parameterise the manifold $W_{\Lambda_{0}}^{u}$ by
$I,\theta,s$
\[
W_{\Lambda_{0}}^{s}=\left\{  \left(  \pi_{X,P_{X}}w_{I}^{\mathrm{s}}\left(
s\right)  ,I,\theta\right)  :I\in\mathbb{R},\text{ }\theta\in\mathbb{T}%
,\,s\in\mathbb{R}\right\}  ,
\]
and parameterise $\Gamma$ by $I,\theta$ as in (\ref{eq:Gamma-def}). We see
that at a point $\left(  \gamma\left(  I\right)  ,\theta\right)  \in\Gamma$
\begin{align}
T_{\left(  \gamma\left(  I\right)  ,\theta\right)  }W_{\Lambda_{0}}^{u}  &
=\mathrm{span}\left\{  \left(  \pi_{X,P_{X}}\frac{d}{du}w_{I}^{\mathrm{u}%
}\left(  u \right) |_{u=0},0,0\right)  ,\left(  \pi_{X,P_{X}}\frac{d}{dI}%
w_{I}^{\mathrm{u}}\left(  0\right)  ,1,0\right)  ,\left(  0,0,0,1\right)
\right\}  ,\label{eq:TWu}\\
T_{\left(  \gamma\left(  I\right)  ,\theta\right)  }W_{\Lambda_{0}}^{s}  &
=\mathrm{span}\left\{  \left(  \pi_{X,P_{X}}\frac{d}{ds}w_{I}^{\mathrm{s}%
}\left(  s\right) |_{s=0},0,0\right)  ,\left(  \pi_{X,P_{X}}\frac{d}{dI}%
w_{I}^{\mathrm{s}}\left(  0\right)  ,1,0\right)  ,\left(  0,0,0,1\right)
\right\}  ,\label{eq:TWs-temp}\\
T_{\left(  \gamma\left(  I\right)  ,\theta\right)  }\Gamma &  =\mathrm{span}%
\left\{  \left(  \pi_{X,P_{X}}\frac{d}{dI}\gamma\left(  I\right)  ,1,0\right)
,\left(  0,0,0,1\right)  \right\}  . \label{eq:TGamma}%
\end{align}
We know that $\gamma\left(  I\right)  $ results from the intersection of
$w_{I}^{\mathrm{u}}\left(  u\right)  $ and $w_{I}^{\mathrm{s}}\left(
s\right)  $ at $u=s=0$
\[
\gamma\left(  I\right)  =w_{I}^{\mathrm{u}}\left(  0\right)  =w_{I}%
^{\mathrm{s}}\left(  0\right)  ,
\]
so%
\begin{equation}
\pi_{X,P_{X}}\frac{d}{dI}\gamma\left(  I\right)  =\pi_{X,P_{X}}\frac{d}%
{dI}w_{I}^{\mathrm{u}}\left(  0\right)  =\pi_{X,P_{X}}\frac{d}{dI}%
w_{I}^{\mathrm{s}}\left(  0\right)  . \label{eq:tangent-vectors-same}%
\end{equation}
We now see that from (\ref{eq:X-PX-Trans}), (\ref{eq:TWu}%
--\ref{eq:tangent-vectors-same}) we have (\ref{eq:scatter-cond-1}%
--\ref{eq:scatter-cond-2}).

We now turn to proving (\ref{eq:scatter-cond-3}--\ref{eq:scatter-cond-4}). For
a fixed $I$ let us take $x=x_{0}\left(  \mathcal{X}\right)  $ where
$\mathcal{X}$ is such that $H_{0}\left(  L_{\mathcal{X}}\right)  =I$. Let us
also consider a fixed $\lambda\in\mathbb{T}$. We see that the unstable and
stable fibres of the point $(x,\lambda)\in\Lambda_{0}$ are parameterised by
$u$ and $s$, respectively, as
\begin{align*}
W_{\left(  x,\lambda\right)  }^{u} &  =\left\{  \left(  \pi_{X,P_{X}}%
w_{I}^{\mathrm{u}}\left(  u\right)  ,I,\theta_{I,\lambda}^{\mathrm{u}}\left(
u\right)  \right)  :u\in\mathbb{R}\right\}  ,\\
W_{\left(  x,\lambda\right)  }^{s} &  =\left\{  \left(  \pi_{X,P_{X}}%
w_{I}^{\mathrm{s}}\left(  s\right)  ,I,\theta_{I,\lambda}^{\mathrm{s}}\left(
s\right)  \right)  :s\in\mathbb{R}\right\}  ,
\end{align*}
where $\theta_{I,\lambda}^{\mathrm{u}},\theta_{I,\lambda}^{\mathrm{s}%
}:\mathbb{R}\rightarrow\mathbb{T}$ are some functions, which parameterise the
fibers along the angle coordinate. Hence for every $I,\theta$
\begin{align}
T_{\left(  \gamma\left(  I\right)  ,\theta\right)  }W_{\left(  x,\theta
\right)  }^{u} &  =\mathrm{span}\left\{  \left(  \pi_{X,P_{X}}\frac{d}%
{du}w_{I}^{\mathrm{u}}\left(  u\right)  |_{u=0},0,\frac{d}{du}\theta
_{I,\theta}^{\mathrm{u}}\left(  u\right)  |_{u=0}\right)  \right\}
,\label{eq:TWux}\\
T_{\left(  \gamma\left(  I\right)  ,\theta\right)  }W_{\left(  x,\theta
\right)  }^{s} &  =\mathrm{span}\left\{  \left(  \pi_{X,P_{X}}\frac{d}%
{ds}w_{I}^{\mathrm{s}}\left(  s\right)  |_{s=0},0,\frac{d}{ds}\theta
_{I,\theta}^{\mathrm{s}}\left(  s\right)  |_{s=0}\right)  \right\}
.\label{eq:TWsx}%
\end{align}
From (\ref{eq:TWs-temp}), (\ref{eq:TGamma}), (\ref{eq:tangent-vectors-same})
and (\ref{eq:TWsx}) we obtain (\ref{eq:scatter-cond-3}). Similarly, from
(\ref{eq:TWu}), (\ref{eq:TGamma}), (\ref{eq:tangent-vectors-same}) and
(\ref{eq:TWux}) we have (\ref{eq:scatter-cond-4}). We have thus shown that
$\Gamma$ is a well defined homoclinic channel.

We now discuss the scattering map (\ref{eq:scatter-unperturbed}) associated with $\Gamma$. For fibers
$W_{\left(  x_{-},\theta_{-}\right)  }^{u}$ and $W_{\left(  x_{+},\theta
_{+}\right)  }^{s}$ to intersect we must have $x_{+}=x_{-}$, since if this was
not the case then the points $x_{-},x_{+}$ would lie on different energy
levels, and their fibres would not meet. This means that\correction{comment 31}{}
\[
\sigma\left(  x,\lambda\right)  =\left(  x,\pi_{\theta}\sigma\left(
x,\lambda\right)  \right)  ,
\]
We now show how to obtain estimates for $\pi_{\theta}\sigma\left(
x,\lambda\right)  $.

For every $\left(  x,\lambda\right)  \in\Lambda_{0}$ we have the homoclinic
orbit established in section \ref{sec:inter} (see in particular Table
\ref{tabl:homoclinic}) with the initial point $x_{0}=x_{0}\left(
x,\lambda\right)  $ lying on the unstable fiber established in Lemma
\ref{lem:fiber-extended}. From (\ref{eq:fibers-extended-2}) it follows that%
\[
\left\vert \pi_{\theta}x_{0}\left(  x,\lambda\right)  -\lambda\right\vert \leq
Mr=3\cdot 3\cdot10^{-8}.
\]
As the point $x_{0}\left(  x,\lambda\right)  $ is iterated by $f_{0}$ the
angle $\theta$ changes. From Remark \ref{rem:time-between-points} we know that
after the full excursion along the homoclinic from Table \ref{tabl:homoclinic}
we return to the neighbourhood of $\Lambda_{0}$ at an angle $\pi_{\theta}%
x_{0}\left(  x,\lambda\right)  +42\cdot h$, where $h=h(x)$ is from the
interval (\ref{eq:h-bound}). Such homoclinic excursion takes five iterates of
$f_{0}$ (see Table \ref{tabl:homoclinic}) so%
\[
\pi_{\theta}f_{0}^{5}\left(  x_{0}\left(  x,\lambda\right)  \right)
=\pi_{\theta}x_{0}\left(  x,\lambda\right)  +42\cdot h.
\]
We know that $f_{0}^{5}\left(  x_{0}\left(  x,\lambda\right)  \right)  $ lies
in the stable fiber of $f_{0}^{5}\left(  \sigma\left(  x,\lambda\right)
\right)  $. We also know that%
\[
f_{0}^{5}\left(  \sigma\left(  x,\lambda\right)  \right)  =f_{0}^{5}\left(
x,\pi_{\theta}\sigma\left(  x,\lambda\right)  \right)  =\left(  x,\pi_{\theta
}\sigma\left(  x,\lambda\right)  +5T(x)\right)  ,
\]
where $T\left(  x\right)  $ is the period of the Lyapunov orbit. From Lemma
\ref{lem:fiber-extended} and the $\mathcal{R}$-symmetry of the system since
$f_{0}^{5}\left(  x_{0}\left(  x,\lambda\right)  \right)  \in W_{f_{0}%
^{5}\left(  \sigma\left(  x,\lambda\right)  \right)  }^{s}$ we have
\[
\left\vert \pi_{\theta}\left(  f_{0}^{5}\left(  x_{0}\left(  x,\lambda\right)
\right)  -f_{0}^{5}\left(  \sigma\left(  x,\lambda\right)  \right)  \right)
\right\vert \leq Mr.
\]
This allows us to obtain the following estimate for the scattering map%
\begin{align}
\pi_{\theta}\sigma\left(  x,\lambda\right)   &  =\pi_{\theta}\sigma\left(
x,\lambda\right)  +5T(x)-5T(x)\nonumber\\
&  =\pi_{\theta}f_{0}^{5}\left(  \sigma\left(  x,\lambda\right)  \right)
-5T(x)\nonumber\\
&  =\pi_{\theta}f_{0}^{5}\left(  x_{0}\left(  x,\lambda\right)  \right)
+\pi_{\theta}\left(  f_{0}^{5}\left(  \sigma\left(  x,\lambda\right)  \right)
-f_{0}^{5}\left(  x_{0}\left(  x,\lambda\right)  \right)  \right)
-5T(x)\nonumber\\
&  \in\pi_{\theta}x_{0}\left(  x,\lambda\right)  +42\cdot h+\left[
-Mr,Mr\right]  -5T(x)\nonumber\\
&  =\lambda+\left(  \pi_{\theta}x_{0}\left(  x,\lambda\right)  -\lambda
\right)  +42\cdot h+\left[  -Mr,Mr\right]  -5T(x)\nonumber\\
&  \in\lambda+\left[  -Mr,Mr\right]  +42\cdot h+\left[  -Mr,Mr\right]
-5T(x)\nonumber\\
&  \in\lambda+[-0.82506903038, -0.82506533656]. \label{eq:sigma-bound-cap}%
\end{align}
We have thus obtained our bound for the scattering map of the unperturbed system.

We have finished our discussion about the unperturbed system. Now we turn to
showing that for sufficiently small $\varepsilon>0$ the manifold $\Lambda_{0}$ persists.

Recall the notation $\tilde{\Lambda}_{L}$ from (\ref{eq:nhim-lyap}), which
stands for our family of Lyapunov orbits in the extended phase space. Recall
also that $\Lambda_{0}=\tilde{\Lambda}_{L}\cap\{Y=0\}$. By Corollary
\ref{cor:persistence} we know that for sufficiently small $\varepsilon$ the
normally hyperbolic invariant manifold $\tilde{\Lambda}_{L}$ is perturbed to a
nearby manifold $\tilde{\Lambda}_{L}^{\varepsilon}$. We thus obtain
$\Lambda_{\varepsilon}:=\tilde{\Lambda}_{L}^{\varepsilon}\cap\{Y=0\}$ as the
perturbation of $\Lambda_{0}$. \correction{comment 30}{}

\cor{
Let $\omega_{\varepsilon}$ be the following symplectic form on $\{Y=0\}$ 
\begin{equation}\omega_{\varepsilon}=
dX\wedge dP_X 
-\frac{\partial H_{\varepsilon}}{\partial X}dX\wedge d\theta
-\frac{\partial H_{\varepsilon}}{\partial P_X}dP_X\wedge d\theta
-\frac{\partial H_{\varepsilon}}{\partial P_Y}dP_Y\wedge d\theta. \label{eq:symplectric-form-3bp}
\end{equation}
The map $\mathcal{P}_{\varepsilon}$ is $\omega_{\varepsilon}$ symplectic. (See Lemma \ref{lem:symplectic} in \ref{sec:poinc-symplectic}.) The matrix $\Omega$ associated with the symplectic form $\omega_{0}$
 is of the form%
\[
\Omega=\left(
\begin{array}
[c]{cccc}%
0 & 1 & 0 & -\frac{\partial H_0}{\partial X}\\
-1 & 0 & 0 & -\frac{\partial H_0}{\partial P_{X}}\\
0 & 0 & 0 & -\frac{\partial H_0}{\partial P_{Y}}\\
\frac{\partial H_0}{\partial X} & \frac{\partial H_0}{\partial P_{X}} &
\frac{\partial H_0}{\partial P_{Y}} & 0
\end{array}
\right)  .
\]
The manifold $\Lambda_{0}$ is parameterised by the function $\lambda
_{0}:\left(  \mathcal{X},\mathcal{\theta}\right)  \mapsto\left(  x_{0}\left(
\mathcal{X}\right)  ,\theta\right)  $, so the tangent space $T\Lambda_{0}$
spans the two vectors $\lambda^{1}:=\frac{\partial\lambda_{0}}{\partial
\mathcal{X}}=(1,0,\frac{ds}{d\mathcal{X}},0)$ and $\lambda^{2}:=\frac
{\partial\lambda_{0}}{\partial\mathcal{\theta}}=\left(  0,0,0,1\right)  .$
Taking the $4\times2$ matrix $C$ with columns consisting of $\lambda^{1}$ and
$\lambda^{2}$, from (\ref{eq:dH-along-x0}) we see that the matrix associated with the
symplectic form $\omega_{0}|_{\Lambda_{0}}$ is equal to
\[
C^{T}\Omega C=\left(
\begin{array}
[c]{cc}%
0 & -\left(  \frac{\partial H_0}{\partial X}+\frac{\partial H_0}{\partial P_{Y}%
}\frac{ds}{d\mathcal{X}}\right)  \\
\frac{\partial H_0}{\partial X}+\frac{\partial H_0}{\partial P_{Y}}\frac
{ds}{d\mathcal{X}} & 0
\end{array}
\right)  =\left(
\begin{array}
[c]{cc}%
0 & -\frac{d}{d\mathcal{X}}H_0\left(  x_{0}\left(  \mathcal{X}\right)  \right)
\\
\frac{d}{d\mathcal{X}}H_0\left(  x_{0}\left(  \mathcal{X}\right)  \right)   & 0
\end{array}
\right)  .
\]
By Lemma \ref{lem:persistence} we see that for $\mathcal{X}$ from the interval (\ref{eq:X-interval})
we have $\frac{d}{d\mathcal{X}}H_0\left(  x_{0}\left(  \mathcal{X}\right)
\right)  \neq0$, hence $\omega_{0}|_{\Lambda_{0}}$ is not degenerate. This means that for sufficiently small $\varepsilon$ the symplectic form $\omega_{\varepsilon}|_{\Lambda_{\varepsilon}}$ is also not degenerate. The map $\mathcal{P}_{\varepsilon}$ is $\omega_{\varepsilon}$ symplectic (see Lemma \ref{lem:symplectic} in \ref{sec:poinc-symplectic}), so $\mathcal{P}_{\varepsilon}\circ\mathcal{P}_{\varepsilon}$ is $\omega_{\varepsilon}|_{\Lambda_{\varepsilon}}$ symplectic on $\Lambda_{\varepsilon}$. 
We can define a measure on $\Lambda_{\varepsilon}$ as 
\begin{equation} \mu(B):= \int_{\left(x^{\ast},0\right)  +\tilde{A}_{0} B } \omega_{\varepsilon}|_{\Lambda_{\varepsilon}} \qquad \mbox{for } B\subset \Lambda_{\varepsilon}.\label{eq:measure-def}\end{equation}
For such measure $f_{\varepsilon}$ defined in (\ref{eq:return-map}) are measure preserving.
}

Since by Theorem \ref{th:persistence} we know
that $\tilde{\Lambda}_{L}^{\varepsilon}$ is a manifold with a boundary that
consists of two, two dimensional invariant KAM tori. These two tori
intersected with $\{Y=0\}$ produce two curves, which are invariant under
$f_{\varepsilon}$. They become boundaries of $\Lambda_{\varepsilon}$. Thus
$\Lambda_{\varepsilon}$ is a normally hyperbolic invariant manifold, with
boundary, for $f_{\varepsilon}$. Similarly, all the two dimensional KAM tori in $\tilde \Lambda^{\varepsilon}_L$ become one dimensional invariant tori in $\Lambda_{\varepsilon}$.

We now turn to validating the assumptions of Theorem \ref{th:shadowing-seq} to
obtain our result. This will be done in the local coordinates given by the
affine change involving $\tilde{A}_{0}$ and $x^{\ast}$, in which our
$f_{\varepsilon}$ is expressed. Recall that in Lemma \ref{lem:Wu-bound-CAP}
and Corollary \ref{cor:ps-der} we have established bounds on the intersection
of the local unstable manifold of a Lyapunov orbit with $\{Y=0\}$. This bound is
valid in the following neighbourhood of $\Lambda_{0}$
\begin{equation}
U=\left[  -r,r\right]  ^{4}\times\mathbb{T}, \label{eq:U-in-proof}%
\end{equation}
where $r=3\cdot10^{-8}$ is the constant from Lemma \ref{lem:Wu-bound-CAP}. We
have performed a computer assisted validation that the constant $L_{g}$ from
(\ref{eq:Lip-g-assumption}) is%
\begin{equation}
L_{g}=90943. \label{eq:Lg-bound}%
\end{equation}
This value was computed using the method described in detail in \ref{sec:appendix}.
The value obtained in (\ref{eq:Lg-bound}) is a large overestimate. By
performing more careful checks, for instance by subdividing $U$ into small
fragments, the bound can be significantly improved. Due to the small size of
the neighbourhood (\ref{eq:U-in-proof}) in which we consider the local
unstable manifold, we see that the constant $C$ from
(\ref{eq:contraction-expansion-bounds}) can is very small:
\begin{equation}
C=r=3\cdot10^{-8}. \label{eq:C-bound}%
\end{equation}
Thus the large value of $L_{g}$ is not a problem for us, since $L_{g}$ enters
condition (\ref{eq:key-assumption-again}) multiplied by $C$. We use the bound
$\bar{m}$ from Remark \ref{rem:Lip-local} to compute
\[
\lambda=\left(  \bar{m}\right)  ^{-30}=\frac{1}{1525.16}.
\]
The power $30$ comes from the fact that to complete a full turn round a
Lyapunov orbit involves $30$ local maps (\ref{eq:fi-def}); see Table
\ref{tabl:Lyap}. We consider the following strips\footnote{The positioning of
the strips was motivated by computing $\sum_{i=0}^{4}\pi_{I}g\left(
0,f_{0}^{i}\left(  x\right)  \right)  $, which represents the change of energy
after the homoclinic excursion along the homoclinic. These five terms in the
sum play the leading role in (\ref{eq:final-check}). In Figure
\ref{fig:energyChange} we provide a plot of computer assisted bounds on $\sum_{i=0}^{4}\pi_{I}g\left(
0,f_{0}^{i}\left(  x\right)  \right)  $, for different choices of $\pi_{\theta
}z$, and place our strips $S^{+},S^{-}$ for reference in the figure.}
\begin{align*}
S^{-}  &  =\Lambda_{0}\cap\left\{  \theta\in\left[
0.65-0.125,0.65+0.125\right]  \right\}  ,\\
S^{+}  &  =\Lambda_{0}\cap\left\{  \theta\in\left[  \pi+0.65-0.125,\pi
+0.65+0.125\right]  \right\}  .
\end{align*}

For each point $z\in S^{+}$ we compute $m=m\left(  z\right)  $ such that
condition (\ref{eq:strip-return-cond}) is fulfilled. Depending on the choice
of $z$ the resulting $m$ can differ. To compute it we used the bound on
$\sigma$ from (\ref{eq:sigma-bound-cap}) and the fact that the inner dynamics
is given by (\ref{eq:inner-dynamics}) with the bound on $T(\mathcal{X})$ from
(\ref{eq:Lyap-period}). We check the assumptions of Theorem \ref{th:strip-up}
by subdividing $S^{+}$ into $25$ fragments along the $\theta$ coordinate, and
validated the assumptions for each fragment independently. For the first three
fragments for which $\theta$ was closest to $\pi+ 0.65-0.125$ it turned out
that a good choice is $m(z)=21$; for three fragments with $\theta$ close to
$\pi+ 0.65+0.125$ we used $m(z)=25$; and for the remaining we took $m(z)=23$.
For $z\in S^{+}$ the point $x\in W_{z}^{u}$ from condition 2. from Theorem
\ref{th:strip-up}, is taken as the first point from homoclinic orbit from
Lemma \ref{lem:homoclinic-bound}. For the alignment of $x$ along $\theta$ we
use the estimate (\ref{eq:fibers-extended-2}) from Lemma
\ref{lem:fiber-extended}. We then validate that%
\begin{equation}
\sum_{i=0}^{m\left(  z\right)  -1}\pi_{I}g\left(  0,f_{0}^{i}\left(  x\right)
\right)  -\frac{1+\lambda}{1-\lambda}L_{g}C>0. \label{eq:final-check}%
\end{equation}
The inequality (\ref{eq:final-check}) is validated with the aid of computer assisted estimates.

\begin{figure}[ptb]
\begin{center}
\includegraphics[height=4.5cm]{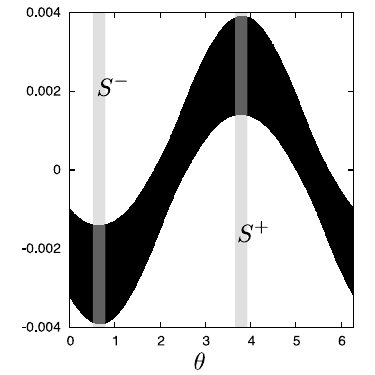}
\end{center}
\caption{The computer assisted bound for the first five terms in the sum in
(\ref{eq:final-check}), depending on $\pi_{\theta} z$.}%
\label{fig:energyChange}%
\end{figure}

In the similar fashion we validate the assumptions of Theorem
\ref{th:strip-down}.

Conditions 1. and 2. from Theorem \ref{th:shadowing-seq} are simple to
validate since for every point $z=\left(  x_{0}\left(  \mathcal{X}\right)
,\lambda\right)  \in\Lambda_{0}$ we have%
\[
\pi_{\theta}f\left(  z\right)  =\lambda+T\left(  \mathcal{X}\right)  ,
\]
where $T(\mathcal{X})$ is the period of the Lyapunov orbit, whose bound is
known to us in (\ref{eq:Lyap-period}).
\end{proof}

\begin{remark}
\label{rem:Oterma-m}Due to the fact that we work with the particular family of
Lyapunov orbits, which correspond to the energy of the comet Oterma, it turned
out that the $m$ used in conditions 1. and 2. from Theorems \ref{th:strip-up}
and \ref{th:strip-down} was a large number. This resulted in the need of large
number of iterates of $f_{0}$ when computing (\ref{eq:final-check}), which
meant long integration time. Also, we needed relatively wide
strips $S^{+},S^{-}$. This meant that we needed many subdivisions of the
strips to perform our validation. The long integration time and the large
number of sets increased the computational time of our proof. By making a more
careful choice of the energy level, one could focus on Lyapunov orbits for
which $m$ would be smaller. We have chosen not to do so to demonstrate that
the method is applicable for the choice of energy dictated by a concrete
physical object.
\end{remark}

\section{Acknowledgements}
We would like to thank the anonymous reviewers for their comments, suggestions, and corrections, which helped us improve our paper. We would also like to thank Piotr Zgliczy\'nski for helpful discussions.

\appendix

\cor{
\section{Symplectic properties of Poincar\'e maps for time dependent Hamiltonian systems in extended phase space}\label{sec:poinc-symplectic}

Consider a time dependent Hamiltonian system, with Hamiltonian $H:\mathbb{R}%
^{2n}\times\mathbb{R}\rightarrow\mathbb{R}$ and assume that $t\rightarrow
H\left(  x,t\right)  $ is $2\pi$ periodic. Consider that $x=\left(
p,q\right)  \in\mathbb{R}^{n}\times\mathbb{R}^{n}$, where $p$ are the
positions and $q$ their associated momenta, and take the standard symplectic
form $\omega=\sum_{i=1}^{n}dp_{i}\wedge dq_{i}$. The vector field in the
extended phase space induced by the Hamiltonian is%
\begin{align}
x^{\prime}  &  =J\nabla_{x}H\left(  x,t\right)  ,\label{eq:ode-time-dep}\\
t^{\prime}  &  =1,\nonumber
\end{align}
where%
\[
J=J_{n}:=\left(
\begin{array}
[c]{cc}%
0 & Id_{n}\\
-Id_{n} & 0
\end{array}
\right)  ,
\]
and $Id_{n}$ is an $n$-dimensional identity matrix.

Let us consider a section $\Sigma:=\left\{  p_{n}=0\right\}  \subset
\mathbb{R}^{2n}\times\mathbb{R}$ and assume that locally we have a well
defined Poincar\'{e} map
\[
P:\Sigma\rightarrow\Sigma.
\]

\begin{lemma}\label{lem:symplectic}
For the symplectic form%
\begin{equation}
\omega_{\Sigma}:=\sum_{i=1}^{n-1}dp_{i}\wedge dq_{i}-\sum_{i=1}^{n-1}%
\frac{\partial H}{\partial p_{i}}dp_{i}\wedge dt-\sum_{i=1}^{n}\frac{\partial
H}{\partial q_{i}}dq_{i}\wedge dt, \label{eq:symplectic-form-in-extended-space-app}
\end{equation}
the map $P$ is $\omega_{\Sigma}$-symplectic. Moreover, $\omega_{\Sigma}$ is nondegenerate.
\end{lemma}

\begin{proof}
We can expand the extended phase space by including an artificial additional
action variable $A$ and defining a Hamiltonian $\mathcal{H}:\mathbb{R}%
^{2\left(  n+1\right)  }\rightarrow\mathbb{R}$ as
\[
\mathcal{H}\left(  p,A,q,t\right)  :=H\left(  p,q,t\right)  +A
\]
and taking the standard symplectic form $\bar{\omega}=\sum_{i=1}^{n}%
dp_{i}\wedge dq_{i}+dA\wedge dt$. The vector field then becomes $\mathcal{F}%
=\mathcal{J}\nabla\mathcal{H},$ where $\mathcal{J}=J_{n+1}$. (Note that by
ignoring the artificial action variable $A$ we see that the solutions of
$\mathcal{F}$ coincide with the solutions of (\ref{eq:ode-time-dep}).)

Let $h\in\mathbb{R}$ be some number, and consider the section $\Sigma
_{h}:=\left\{  p_{n}=0,\mathcal{H}=h\right\}  \subset\mathbb{R}^{2\left(
n+1\right)  }$ and a Poincar\'{e} map $\mathcal{P}_{h}:\Sigma_{h}%
\rightarrow\Sigma_{h}$. The map $\mathcal{P}_{h}$ is symplectic with the
symplectic form $\omega_{h}=\sum_{i=1}^{n-1}dp_{i}\wedge dq_{i}+dA\wedge dt$.
We can parameterise the section $\Sigma_{h}$ by using coordinates $\left(
p_{1},\ldots,p_{n-1},q,t\right)  $; this is because
\[
\Sigma_{h}=\left\{  \left(  p_{1},\ldots,p_{n-1},0,A,q,t\right)  :A=h-H\left(
p_{1},\ldots,p_{n-1},0,q,t\right)  \right\}  .
\]

Since on $\Sigma_{h}$ we have $A=h-H\left(  p_{1},\ldots,p_{n-1},0,q,t\right)
$, in the coordinates $\left(  p_{1},\ldots,p_{n-1},q,t\right)  $, we have%
\[
\omega_{h}=\omega_{\Sigma}.
\]
We also observe that in the coordinates $\left(  p_{1},\ldots,p_{n-1}%
,q,t\right)  $ both the Poincar\'{e} map $\mathcal{P}_{h}$ as well as the
symplectic form $\omega_{h}$ do not depend on $h$. Moreover, in these
coordinates $\mathcal{P}_{h}=P$. We thus obtain that $P$ is $\omega_{\Sigma}%
$-symplectic, as required. 

The matrix $\Omega$ associated to $\omega_{\Sigma}$, (i.e. $\omega_{\Sigma
}\left(  v,w\right)  =v^{\top}\Omega w$) is of the form%
\begin{equation}
\Omega=\left(
\begin{array}
[c]{cccccccc}%
0 & \cdots & 0 & 1 &  & 0 & 0 & -\frac{\partial H}{\partial p_{1}}\\
\vdots &  & \vdots &  & \ddots &  & \vdots & \vdots\\
0 & \cdots & 0 & 0 & \cdots & 1 & 0 & -\frac{\partial H}{\partial p_{n-1}}\\
-1 &  & 0 & 0 & \cdots & 0 & 0 & -\frac{\partial H}{\partial q_{1}}\\
& \ddots &  & \vdots &  & \vdots & \vdots & \vdots\\
0 &  & -1 & 0 & \cdots & 0 & 0 & -\frac{\partial H}{\partial q_{n-1}}\\
0 & \cdots & 0 & 0 & \cdots & 0 & 0 & -\frac{\partial H}{\partial q_{n}}\\
\frac{\partial H}{\partial p_{1}} & \cdots & \frac{\partial H}{\partial
p_{n-1}} & \frac{\partial H}{\partial q_{1}} & \cdots & \frac{\partial
H}{\partial q_{n-1}} & \frac{\partial H}{\partial q_{n}} & 0
\end{array}
\right)  .\label{eq:sym-matrix}
\end{equation}
Since the Poincar\'e map is well defined, we have $\frac{\partial H}{\partial q_{n}}\neq0$. The fact that for all
$w$ we have $v^{\top}\Omega w=0$ implies that $v=0$, so $\omega_{\Sigma}$ is
nondegenerate, as required.
\end{proof}

}

\section{Proof of Lemma \ref{lem:expansion-technical} \label{sec:expansion-technical}}
\begin{proof}
Take $v\in Q^{+}\left(  0\right)  \cap B$. Since $a_{i}\in\left(  0,1\right)
$ we see that $\left\vert v_{1}\right\vert \geq\left\vert v_{i}\right\vert
/a_{i-1}>\left\vert v_{i}\right\vert $, for $i=2,3,4$. Since $f\left(
0\right)  =0$ from the mean value theorem we obtain
\begin{align*}
\left\Vert f\left(  v\right)  \right\Vert _{\max}  &  =\left\Vert f\left(
v\right)  -f\left(  0\right)  \right\Vert _{\max}\\
&  \in\left\Vert \left[  Df\left(  B\right)  \right]  v\right\Vert _{\max}\\
&  =\left\Vert \left(  A_{11}v_{1}+A_{12}\left(  v_{2},v_{3},v_{4}\right)
,A_{21}v_{1}+A_{22}\left(  v_{2},v_{3},v_{4}\right)  \right)  \right\Vert
_{\max}\\
&  \geq\left\vert A_{11}v_{1}+A_{12}\left(  v_{2},v_{3},v_{4}\right)
\right\vert \\
&  \geq\left(  c-\left\Vert A_{12}\right\Vert _{\max}\max\left(  a_{1}%
,a_{2},a_{3}\right)  \right)  \left\vert v_{1}\right\vert \\
&  =\left(  c-\left\Vert A_{12}\right\Vert _{\max}\max\left(  a_{1}%
,a_{2},a_{3}\right)  \right)  \left\Vert v\right\Vert _{\max},
\end{align*}
as required.
\end{proof}

\section{Proof of Lemma \ref{lem:Lip-technical} \label{sec:Lip-technical}}
\begin{proof}
Let $v=\left(  v_{1},v_{2}\right)  $ where $v_{1}=\pi_{1}v$ and $v_{2}%
=\pi_{2,3,4}v$. Since $v\in Q^{+}\left(  0\right)  $ we see that%
\[
\left\vert v_{1}\right\vert \geq a^{-1}\left\Vert v_{2}\right\Vert _{\max}.
\]
From the mean value theorem we obtain
\begin{align*}
\left\Vert f\left(  v\right)  \right\Vert _{\max} &  \in\left\Vert \left[
Df\left(  B\right)  \right]  v\right\Vert _{\max}\\
&  =\left\Vert \left(  A_{11}v_{1}+A_{12}v_{2},A_{21}v_{1}+A_{22}v_{2}\right)
\right\Vert _{\max}\\
&  \leq\max\left(  \left\vert A_{11}\right\vert \left\vert v_{1}\right\vert
+\left\Vert A_{12}\right\Vert _{\max}\left\Vert v_{2}\right\Vert _{\max
},\left\Vert A_{21}\right\Vert _{\max}\left\vert v_{1}\right\vert +\left\Vert
A_{22}\right\Vert _{\max}\left\Vert v_{2}\right\Vert _{\max}\right)  \\
&  \leq\max\left(  \left\vert A_{11}\right\vert +a\left\Vert A_{12}\right\Vert
_{\max},\left\Vert A_{21}\right\Vert _{\max}+a\left\Vert A_{22}\right\Vert
_{\max}\right)  \left\vert v_{1}\right\vert \\
&  =\max\left(  \left\vert A_{11}\right\vert +a\left\Vert A_{12}\right\Vert
_{\max},\left\Vert A_{21}\right\Vert _{\max}+a\left\Vert A_{22}\right\Vert
_{\max}\right)  \left\Vert v\right\Vert _{\max}%
\end{align*}
as required.
\end{proof}

\section{Lipschitz bounds for the perturbation term\label{sec:appendix}}

Here we give a method, with which we can check (\ref{eq:Lip-g-assumption}) in
the case when $I(x)$ is defined as (\ref{eq:energy-in-proof}). Below we start
with Lemma \ref{lem:Lg-bound-1}, that can be applied to achieve this. We have
found though that in our particular case of the PER3BP, due to long
integration times, a direct application of Lemma \ref{lem:Lg-bound-1} leads to
overestimates, which were too large for our needs. We therefore follow with
Lemma \ref{lem:Lg-bound-2}, which can be used for an inductive computation of
$L_{g}$ by expressing $f_{\varepsilon}$ as a composition of maps. This allowed
us to avoid long integration times and improved the estimate.

\begin{lemma}
\label{lem:Lg-bound-1} Let $D^{2}I(x)$ stand for the Hessian of $I$ at $x$.
Consider $x_{0},x_{1}\in\mathbb{R}^{3}\times\mathbb{S}$ and let $\mathbf{x}%
\subset\mathbb{R}^{3}\times\mathbb{S}$ be a convex set which contains $x_{0}$
and $x_{1}$. Let $\mathbf{v}\subset\mathbb{R}^{4}$ be the following interval
enclosure
\[
\mathbf{v}^{\top}=\left[  \left(  \frac{\partial f_{0}}{\partial\varepsilon
}\left(  \mathbf{x}\right)  \right)  ^{\top}D^{2}I\left(  f_{0}\left(
\mathbf{x}\right)  \right)  Df_{0}\left(  \mathbf{x}\right)  +DI\left(
f_{0}\left(  \mathbf{x}\right)  \right)  \frac{\partial^{2}f_{0}}{\partial
x\partial\varepsilon}\left(  \mathbf{x}\right)  \right]  .
\]
If $\left\Vert \mathbf{v}\right\Vert \leq L_{g}$ then%
\[
\left\vert \pi_{I}g\left(  0,x_{1}\right)  -\pi_{I}g\left(  0,x_{0}\right)
\right\vert \leq L_{g}\left\Vert x_{1}-x_{0}\right\Vert .
\]

\end{lemma}

\begin{proof}
Let
\[
x_{s}:=x_{0}+s\left(  x_{1}-x_{0}\right)  .
\]
From (\ref{eq:gI-at-zero}) we obtain%
\begin{align*}
&  \pi_{I}g\left(  0,x_{1}\right)  -\pi_{I}g\left(  0,x_{0}\right) \\
&  =\int_{0}^{1}\frac{d}{ds}\nabla I\left(  f_{0}\left(  x_{s}\right)
\right)  \cdot\frac{\partial f_{0}}{\partial\varepsilon}\left(  x_{s}\right)
ds\\
&  =\int_{0}^{1}D^{2}I\left(  f_{0}\left(  x_{s}\right)  \right)
Df_{0}\left(  x_{s}\right)  \left(  x_{1}-x_{0}\right)  \cdot\frac{\partial
f_{0}}{\partial\varepsilon}\left(  x_{s}\right)  +\nabla I\left(  f_{0}\left(
x_{s}\right)  \right)  \cdot\frac{\partial^{2}f_{0}}{\partial x\partial
\varepsilon}\left(  x_{s}\right)  \left(  x_{1}-x_{0}\right)  ds\\
&  =\int_{0}^{1}\frac{\partial f_{0}}{\partial\varepsilon}\left(
x_{s}\right)  \cdot D^{2}I\left(  f_{0}\left(  x_{s}\right)  \right)
Df_{0}\left(  x_{s}\right)  \left(  x_{1}-x_{0}\right)  +\nabla I\left(
f_{0}\left(  x_{s}\right)  \right)  \cdot\frac{\partial^{2}f_{0}}{\partial
x\partial\varepsilon}\left(  x_{s}\right)  \left(  x_{1}-x_{0}\right)  ds\\
&  =\left(  \int_{0}^{1}\left(  \frac{\partial f_{0}}{\partial\varepsilon
}\left(  x_{s}\right)  \right)  ^{\top}D^{2}I\left(  f_{0}\left(
x_{s}\right)  \right)  Df_{0}\left(  x_{s}\right)  +DI\left(  f_{0}\left(
x_{s}\right)  \right)  \frac{\partial^{2}f_{0}}{\partial x\partial\varepsilon
}\left(  x_{s}\right)  ds\right)  \left(  x_{1}-x_{0}\right) \\
&  \in\mathbf{v}\cdot\left(  x_{1}-x_{0}\right)  ,
\end{align*}
so%
\[
\left\vert \pi_{I}g\left(  0,x_{1}\right)  -\pi_{I}g\left(  0,x_{0}\right)
\right\vert \leq\left\Vert \mathbf{v}\right\Vert \left\Vert x_{1}%
-x_{0}\right\Vert \leq L_{g}\left\Vert x_{1}-x_{0}\right\Vert ,
\]
as required.
\end{proof}

We now consider the case where $f$ is a composition of two functions
$f=f_{\varepsilon}^{1}\circ f_{\varepsilon}^{2}$. Our objective will be to
compute Lipschitz bound in terms of $x$ for
\begin{equation}
g\left(  \varepsilon,x\right)  :=\frac{1}{\varepsilon}\left(  I\circ
f_{\varepsilon}^{2}\circ f_{\varepsilon}^{1}\left(  \varepsilon,x\right)
-I(x)\right)  . \label{eq:g-for-f1-f2}%
\end{equation}
The following lemma gives the bound for $\pi_{I}g\left(  0,x\right)  $ from
bounds for $f_{\varepsilon}^{1}$ and $f_{\varepsilon}^{2}.$

\begin{lemma}
\label{lem:Lg-bound-2}Assume that%
\begin{align*}
f_{\varepsilon}  &  =f_{\varepsilon}^{2}\circ f_{\varepsilon}^{1},\\
g\left(  \varepsilon,x\right)   &  =\frac{1}{\varepsilon}\left(
f_{\varepsilon}\left(  x\right)  -f_{0}\left(  x\right)  \right)  ,\\
g_{1}\left(  \varepsilon,x\right)   &  =\frac{1}{\varepsilon}\left(
f_{\varepsilon}^{1}\left(  x\right)  -f_{0}^{1}\left(  x\right)  \right)  ,\\
g_{2}\left(  \varepsilon,x\right)   &  =\frac{1}{\varepsilon}\left(
f_{\varepsilon}^{2}\left(  x\right)  -f_{0}^{2}\left(  x\right)  \right)  ,
\end{align*}
and that
\begin{align*}
\left\Vert \pi_{I}\left(  g_{1}\left(  0,x_{1}\right)  -g_{1}\left(
0,x_{0}\right)  \right)  \right\Vert  &  \leq L_{g}^{1}\left\Vert x_{1}%
-x_{0}\right\Vert ,\\
\left\Vert \pi_{I}\left(  g_{2}\left(  0,x_{1}\right)  -g_{2}\left(
0,x_{0}\right)  \right)  \right\Vert  &  \leq L_{g}^{2}\left\Vert x_{1}%
-x_{0}\right\Vert ,\\
\left\Vert f_{0}^{1}\left(  x_{1}\right)  -f_{0}^{1}\left(  x_{0}\right)
\right\Vert  &  \leq L_{f}^{1}\left\Vert x_{1}-x_{0}\right\Vert .
\end{align*}
Then%
\[
\left\Vert \pi_{I}\left(  g\left(  0,x_{1}\right)  -g\left(  0,x_{0}\right)
\right)  \right\Vert \leq\left(  L_{g}^{2}L_{f}^{1}+L_{g}^{1}\right)
\left\Vert x_{1}-x_{0}\right\Vert .
\]

\end{lemma}

\begin{proof}
Consider fixed $x_{1},x_{0}$. Then%
\begin{align*}
&  \left\Vert \pi_{I}g\left(  \varepsilon,x_{1}\right)  -\pi_{I}g\left(
\varepsilon,x_{0}\right)  \right\Vert \\
&  =\left\Vert \frac{1}{\varepsilon}\left(  I\circ f_{\varepsilon}^{2}\circ
f_{\varepsilon}^{1}\left(  x_{1}\right)  -I\left(  x_{1}\right)  -\left(
I\circ f_{\varepsilon}^{2}\circ f_{\varepsilon}^{1}\left(  x_{0}\right)
-I\left(  x_{0}\right)  \right)  \right)  \right\Vert \\
&  \leq\left\Vert \frac{1}{\varepsilon}\left(  I\circ f_{\varepsilon}^{2}\circ
f_{\varepsilon}^{1}\left(  x_{1}\right)  -I\circ f_{\varepsilon}^{1}\left(
x_{1}\right)  -\left(  I\circ f_{\varepsilon}^{2}\circ f_{\varepsilon}%
^{1}\left(  x_{0}\right)  -I\circ f_{\varepsilon}^{1}\left(  x_{0}\right)
\right)  \right)  \right\Vert \\
&  \quad+\left\Vert \frac{1}{\varepsilon}\left(  I\circ f_{\varepsilon}%
^{1}\left(  x_{1}\right)  -I\left(  x_{1}\right)  -\left(  I\circ
f_{\varepsilon}^{1}\left(  x_{0}\right)  -I\left(  x_{0}\right)  \right)
\right)  \right\Vert \\
&  =\left\Vert \pi_{I}g_{2}\left(  \varepsilon,f_{\varepsilon}^{1}\left(
x_{1}\right)  \right)  -\pi_{I}g_{2}\left(  \varepsilon,f_{\varepsilon}%
^{1}\left(  x_{0}\right)  \right)  \right\Vert +\left\Vert \pi_{I}g_{1}\left(
\varepsilon,x_{1}\right)  -\pi_{I}g_{1}\left(  \varepsilon,x_{0}\right)
\right\Vert \\
&  \leq L_{g}^{2}\left\Vert f_{\varepsilon}^{1}\left(  x_{1}\right)
+f_{\varepsilon}^{1}\left(  x_{0}\right)  \right\Vert +L_{g}^{1}\left\Vert
x_{1}+x_{0}\right\Vert +o\left(  \varepsilon\right) \\
&  \leq\left(  L_{g}^{2}L_{f}^{1}+L_{g}^{1}\right)  \left\Vert x_{1}%
+x_{0}\right\Vert +o\left(  \varepsilon\right)
\end{align*}
and the result follows by passing with $\varepsilon$ to zero.
\end{proof}


The bound (\ref{eq:Lip-g-assumption}) computed in the proof (see section
\ref{sec:main-proof}) is for the Poincar\'{e} map $f_{\varepsilon}$ defined in
(\ref{eq:return-map}). This can be computed by considering
\begin{equation}
f_{\varepsilon}=f_{\varepsilon,N}\circ\ldots\circ f_{\varepsilon
,1},\label{eq:f-epsilon-split}%
\end{equation}
where $f_{\varepsilon,i}$ are local maps of the form
\begin{align*}
f_{\varepsilon,i}\left(  x\right)    & =A_{i}^{-1}\left(  \Phi_{t}%
^{\varepsilon}\left(  x_{i-1}+A_{i-1}x\right)  -x_{i}\right)  \qquad\text{for
}i=1,\ldots,N-1,\\
f_{\varepsilon,N}\left(  x\right)    & =A_{0}^{-1}\left(  \mathcal{P}%
^{\varepsilon}\left(  x_{N-1}+A_{N-1}x\right)  -x_{0}\right)  .
\end{align*}
In the inductive application of Lemma \ref{lem:Lg-bound-2} we can use the
bound $\bar{m}$ from Remark \ref{rem:Lip-local} and take $L_{f}^{1}=\bar
{m}^{k}$ at the $k$-th inductive step.

The energy after an iterate of a local map is%
\[
I\left(  f_{\varepsilon,i}\left(  x\right)  \right)  =H_{0}\left(  \Phi
_{t}^{\varepsilon}\left(  x_{i-1}+A_{i-1}x\right)  \right)  .
\]
Since
\begin{align*}
&  \frac{d}{dx}\frac{d}{d\varepsilon}H_{0}\left(  \Phi_{t}^{\varepsilon
}\left(  x_{i-1}+A_{i-1}x\right)  \right)  \\
&  =\frac{d}{dx}\left(  \nabla H_{0}\left(  \Phi_{t}^{\varepsilon}\left(
x_{i-1}+A_{i-1}x\right)  \right)  \cdot\frac{\partial\Phi_{t}^{\varepsilon}%
}{\partial\varepsilon}\left(  x_{i-1}+A_{i-1}x\right)  \right)  \\
&  =\left(  \frac{\partial\Phi_{t}^{\varepsilon}}{\partial\varepsilon}\left(
x_{i-1}+A_{i-1}x\right)  \right)  ^{\top}D^{2}H_{0}\left(  \Phi_{t}%
^{\varepsilon}\left(  x_{i-1}+A_{i-1}x\right)  \right)  \frac{\partial\Phi
_{t}^{\varepsilon}}{\partial x}\left(  x_{i-1}+A_{i-1}x\right)  A_{i-1}\\
&  \quad+DH_{0}\left(  \Phi_{t}^{\varepsilon}\left(  x_{i-1}+A_{i-1}x\right)
\right)  \frac{\partial^{2}\Phi_{t}^{\varepsilon}}{\partial x\partial
\varepsilon}\left(  x_{i-1}+A_{i-1}x\right)  A_{i-1}%
\end{align*}
the $\mathbf{v}^{\top}$ used for the computation of $L_{g}$ for a local map
$f_{\varepsilon,i}$ by means of Lemma \ref{lem:Lg-bound-1} is computed as%
\begin{equation}
\mathbf{v}^{\top}=\left[  \left(  \frac{\partial\Phi_{t}^{\varepsilon}%
}{\partial\varepsilon}\left(  \mathbf{x}\right)  \right)  ^{\top}D^{2}%
H_{0}\left(  \Phi_{t}^{\varepsilon}\left(  \mathbf{x}\right)  \right)
\frac{\partial\Phi_{t}^{\varepsilon}}{\partial x}\left(  \mathbf{x}\right)
A_{i-1}+DH_{0}\left(  \Phi_{t}^{\varepsilon}\left(  \mathbf{x}\right)
\right)  \frac{\partial^{2}\Phi_{t}^{\varepsilon}}{\partial x\partial
\varepsilon}\left(  \mathbf{x}\right)  A_{i-1}\right]
.\label{eq:v-for-local-maps}%
\end{equation}
(For the last local map $f_{\varepsilon,N}$ we take $\mathcal{P}^{\varepsilon
}$ instead of $\Phi_{t}^{\varepsilon}$ in (\ref{eq:v-for-local-maps}).)



\bibliographystyle{elsarticle-num} 
\bibliography{ref}

\begin{thebibliography}{10}
\expandafter\ifx\csname url\endcsname\relax
  \def\url#1{\texttt{#1}}\fi
\expandafter\ifx\csname urlprefix\endcsname\relax\def\urlprefix{URL }\fi
\expandafter\ifx\csname href\endcsname\relax
  \def\href#1#2{#2} \def\path#1{#1}\fi

\bibitem{MR0163026}
V.~I. Arnol'd, Instability of dynamical systems with many degrees of freedom,
  Dokl. Akad. Nauk SSSR 156 (1964) 9--12.

\bibitem{DLS1}
A.~Delshams, R.~de~la Llave, T.~M. Seara,
  \href{https://doi-org.ezproxy.fau.edu/10.1016/j.aim.2007.08.014}{Geometric
  properties of the scattering map of a normally hyperbolic invariant
  manifold}, Adv. Math. 217~(3) (2008) 1096--1153.
\newline\urlprefix\url{https://doi-org.ezproxy.fau.edu/10.1016/j.aim.2007.08.014}

\bibitem{MR1373998}
A.~Delshams, R.~Ram\'{\i}rez-Ros,
  \href{https://doi.org/10.1088/0951-7715/9/1/001}{Poincar\'{e}-{M}el'nikov-{A}rnol'd
  method for analytic planar maps}, Nonlinearity 9~(1) (1996) 1--26.
\newblock \href {https://doi.org/10.1088/0951-7715/9/1/001}
  {\path{doi:10.1088/0951-7715/9/1/001}}.
\newline\urlprefix\url{https://doi.org/10.1088/0951-7715/9/1/001}

\bibitem{GLS}
M.~Gidea, R.~de~la Llave, T.~M-Seara,
  \href{https://doi.org/10.1002/cpa.21856}{A general mechanism of diffusion in
  {H}amiltonian systems: qualitative results}, Comm. Pure Appl. Math. 73~(1)
  (2020) 150--209.
\newblock \href {https://doi.org/10.1002/cpa.21856}
  {\path{doi:10.1002/cpa.21856}}.
\newline\urlprefix\url{https://doi.org/10.1002/cpa.21856}

\bibitem{MR4160091}
M.~Gidea, R.~de~la Llave, T.~M. Seara,
  \href{https://doi.org/10.3934/dcds.2020166}{A general mechanism of
  instability in {H}amiltonian systems: skipping along a normally hyperbolic
  invariant manifold}, Discrete Contin. Dyn. Syst. 40~(12) (2020) 6795--6813.
\newblock \href {https://doi.org/10.3934/dcds.2020166}
  {\path{doi:10.3934/dcds.2020166}}.
\newline\urlprefix\url{https://doi.org/10.3934/dcds.2020166}

\bibitem{Jorge}
M.~J. Capiński, J.~Gonzalez, J.-P. Marco, J.~D. {Mireles James},
  \href{https://www.sciencedirect.com/science/article/pii/S1007570421002823}{Computer
  assisted proof of drift orbits along normally hyperbolic manifolds},
  Communications in Nonlinear Science and Numerical Simulation 106 (2022)
  105970.
\newblock \href {https://doi.org/https://doi.org/10.1016/j.cnsns.2021.105970}
  {\path{doi:https://doi.org/10.1016/j.cnsns.2021.105970}}.
\newline\urlprefix\url{https://www.sciencedirect.com/science/article/pii/S1007570421002823}

\bibitem{MR3927089}
A.~Delshams, V.~Kaloshin, A.~de~la Rosa, T.~M. Seara,
  \href{https://doi.org/10.1007/s00220-018-3248-z}{Global instability in the
  restricted planar elliptic three body problem}, Comm. Math. Phys. 366~(3)
  (2019) 1173--1228.
\newblock \href {https://doi.org/10.1007/s00220-018-3248-z}
  {\path{doi:10.1007/s00220-018-3248-z}}.
\newline\urlprefix\url{https://doi.org/10.1007/s00220-018-3248-z}

\bibitem{CG}
M.~J. Capiński, M.~Gidea,
  \href{https://onlinelibrary.wiley.com/doi/abs/10.1002/cpa.22014}{Arnold
  diffusion, quantitative estimates, and stochastic behavior in the three-body
  problem}, Communications on Pure and Applied Mathematics n/a~(n/a).
\newblock \href {https://doi.org/https://doi.org/10.1002/cpa.22014}
  {\path{doi:https://doi.org/10.1002/cpa.22014}}.
\newline\urlprefix\url{https://onlinelibrary.wiley.com/doi/abs/10.1002/cpa.22014}

\bibitem{MR2785975}
M.~J. Capi\'{n}ski, P.~Zgliczy\'{n}ski,
  \href{https://doi.org/10.1088/0951-7715/24/5/002}{Transition tori in the
  planar restricted elliptic three-body problem}, Nonlinearity 24~(5) (2011)
  1395--1432.
\newblock \href {https://doi.org/10.1088/0951-7715/24/5/002}
  {\path{doi:10.1088/0951-7715/24/5/002}}.
\newline\urlprefix\url{https://doi.org/10.1088/0951-7715/24/5/002}

\bibitem{MR3604613}
M.~J. Capi\'{n}ski, M.~Gidea, R.~de~la Llave,
  \href{https://doi.org/10.1088/1361-6544/30/1/329}{Arnold diffusion in the
  planar elliptic restricted three-body problem: mechanism and numerical
  verification}, Nonlinearity 30~(1) (2017) 329--360.
\newblock \href {https://doi.org/10.1088/1361-6544/30/1/329}
  {\path{doi:10.1088/1361-6544/30/1/329}}.
\newline\urlprefix\url{https://doi.org/10.1088/1361-6544/30/1/329}

\bibitem{capd}
T.~Kapela, M.~Mrozek, D.~Wilczak, P.~Zgliczyński,
  \href{https://www.sciencedirect.com/science/article/pii/S1007570420304081}{Capd::dynsys:
  a flexible c++ toolbox for rigorous numerical analysis of dynamical systems},
  Communications in Nonlinear Science and Numerical Simulation (2020)
  105578\href {https://doi.org/https://doi.org/10.1016/j.cnsns.2020.105578}
  {\path{doi:https://doi.org/10.1016/j.cnsns.2020.105578}}.
\newline\urlprefix\url{https://www.sciencedirect.com/science/article/pii/S1007570420304081}

\bibitem{MR1318950}
G.~Alefeld, Inclusion methods for systems of nonlinear equations---the interval
  {N}ewton method and modifications, in: Topics in validated computations
  ({O}ldenburg, 1993), Vol.~5 of Stud. Comput. Math., North-Holland, Amsterdam,
  1994, pp. 7--26.

\bibitem{MR292101}
M.~W. Hirsch, C.~C. Pugh, M.~Shub,
  \href{https://doi.org/10.1090/S0002-9904-1970-12537-X}{Invariant manifolds},
  Bull. Amer. Math. Soc. 76 (1970) 1015--1019.
\newblock \href {https://doi.org/10.1090/S0002-9904-1970-12537-X}
  {\path{doi:10.1090/S0002-9904-1970-12537-X}}.
\newline\urlprefix\url{https://doi.org/10.1090/S0002-9904-1970-12537-X}

\bibitem{MR426056}
N.~Fenichel, \href{https://doi.org/10.1512/iumj.1977.26.26006}{Asymptotic
  stability with rate conditions. {II}}, Indiana Univ. Math. J. 26~(1) (1977)
  81--93.
\newblock \href {https://doi.org/10.1512/iumj.1977.26.26006}
  {\path{doi:10.1512/iumj.1977.26.26006}}.
\newline\urlprefix\url{https://doi.org/10.1512/iumj.1977.26.26006}

\bibitem{MR343314}
N.~Fenichel, \href{https://doi.org/10.1090/S0002-9904-1974-13498-1}{Asymptotic
  stability with rate conditions for dynamical systems}, Bull. Amer. Math. Soc.
  80 (1974) 346--349.
\newblock \href {https://doi.org/10.1090/S0002-9904-1974-13498-1}
  {\path{doi:10.1090/S0002-9904-1974-13498-1}}.
\newline\urlprefix\url{https://doi.org/10.1090/S0002-9904-1974-13498-1}

\bibitem{S}
V.~Szebehely, Theory of Orbits: The Restricted Problem of Three Bodies,
  Academic Press, 1967.

\bibitem{MR1765636}
W.~S. Koon, M.~W. Lo, J.~E. Marsden, S.~D. Ross,
  \href{https://doi.org/10.1063/1.166509}{Heteroclinic connections between
  periodic orbits and resonance transitions in celestial mechanics}, Chaos
  10~(2) (2000) 427--469.
\newblock \href {https://doi.org/10.1063/1.166509}
  {\path{doi:10.1063/1.166509}}.
\newline\urlprefix\url{https://doi.org/10.1063/1.166509}

\bibitem{MR3467671}
A.~Haro, M.~Canadell, J.-L. Figueras, A.~Luque, J.-M. Mondelo,
  \href{https://doi.org/10.1007/978-3-319-29662-3}{The parameterization method
  for invariant manifolds}, Vol. 195 of Applied Mathematical Sciences,
  Springer, [Cham], 2016, from rigorous results to effective computations.
\newblock \href {https://doi.org/10.1007/978-3-319-29662-3}
  {\path{doi:10.1007/978-3-319-29662-3}}.
\newline\urlprefix\url{https://doi.org/10.1007/978-3-319-29662-3}

\bibitem{MR3309008}
M.~Canadell, A.~Haro,
  \href{https://doi.org/10.1007/978-3-319-06953-1_9}{Parameterization method
  for computing quasi-periodic reducible normally hyperbolic invariant tori},
  in: Advances in differential equations and applications, Vol.~4 of SEMA SIMAI
  Springer Ser., Springer, Cham, 2014, pp. 85--94.
\newblock \href {https://doi.org/10.1007/978-3-319-06953-1\_9}
  {\path{doi:10.1007/978-3-319-06953-1\_9}}.
\newline\urlprefix\url{https://doi.org/10.1007/978-3-319-06953-1_9}

\bibitem{MR3397322}
M.~J. Capi\'{n}ski, P.~Zgliczy\'{n}ski,
  \href{https://doi.org/10.1016/j.jde.2015.07.020}{Geometric proof for normally
  hyperbolic invariant manifolds}, J. Differential Equations 259~(11) (2015)
  6215--6286.
\newblock \href {https://doi.org/10.1016/j.jde.2015.07.020}
  {\path{doi:10.1016/j.jde.2015.07.020}}.
\newline\urlprefix\url{https://doi.org/10.1016/j.jde.2015.07.020}

\bibitem{CALLEJA2022106099}
R.~Calleja, A.~Celletti, J.~Gimeno, R.~{de la Llave},
  \href{https://www.sciencedirect.com/science/article/pii/S1007570421004111}{Kam
  quasi-periodic tori for the dissipative spin–orbit problem}, Communications
  in Nonlinear Science and Numerical Simulation 106 (2022) 106099.
\newblock \href {https://doi.org/https://doi.org/10.1016/j.cnsns.2021.106099}
  {\path{doi:https://doi.org/10.1016/j.cnsns.2021.106099}}.
\newline\urlprefix\url{https://www.sciencedirect.com/science/article/pii/S1007570421004111}

\bibitem{MR3032848}
M.~J. Capi\'{n}ski, \href{https://doi.org/10.1137/110847366}{Computer assisted
  existence proofs of {L}yapunov orbits at {$L_2$} and transversal
  intersections of invariant manifolds in the {J}upiter-{S}un {PCR}3{BP}}, SIAM
  J. Appl. Dyn. Syst. 11~(4) (2012) 1723--1753.
\newblock \href {https://doi.org/10.1137/110847366}
  {\path{doi:10.1137/110847366}}.
\newline\urlprefix\url{https://doi.org/10.1137/110847366}

\bibitem{PoincWZ}
T.~Kapela, D.~Wilczak, P.~Zgliczy\'nski, Recent advances in rigorous
  computation of poincar\'e maps, https://arxiv.org/abs/2104.08046.

\end{thebibliography}





\end{document}